\documentclass[12pt]{amsart}      %{article} was 12pt latex e
\usepackage{amssymb}
\usepackage{eucal}
\usepackage{graphicx}
\usepackage{amsmath}
\usepackage{amscd}
\usepackage[all]{xy}           %xypic macro for latex2.09
\usepackage[active]{srcltx} %SRC Specials for DVI Searching
\usepackage{longtable, lscape}
\usepackage{rotating}

\usepackage{amsfonts,latexsym}
\usepackage{xspace}
\usepackage{epsfig}
\usepackage{float}
\usepackage{color}
\usepackage{fancybox}
\usepackage{colordvi}
\usepackage{multicol}
\usepackage[colorlinks,final,backref=page,hyperindex,hypertex]{hyperref}

\usepackage{tikz}

\newtheorem{theorem}{Theorem}[section]
\newtheorem{proposition}[theorem]{Proposition}
\newtheorem{lemma}[theorem]{Lemma}
\newtheorem{corollary}[theorem]{Corollary}
\theoremstyle{definition}
\newtheorem{definition}[theorem]{Definition}
\newtheorem{example}[theorem]{Example}

\theoremstyle{remark}

\newcommand{\nc}{\newcommand}

\newcommand{\delete}[1]{}

\delete{
\nc{\mlabel}[1]{\label{#1}  % Use the next two lines to show names
{\hfill \hspace{1cm}{\bf{{\ }\hfill(#1)}}}}
\nc{\mcite}[1]{\cite{#1}{{\bf{{\ }(#1)}}}}  % Use this lines to show names
\nc{\mref}[1]{\ref{#1}{{\bf{{\ }(#1)}}}}  % Use this lines to show names
\nc{\mbibitem}[1]{\bibitem[\bf #1]{#1}} % Use this to show name
}

\nc{\mlabel}[1]{\label{#1}}  % Use this to suppress names
\nc{\mcite}[1]{\cite{#1}}  % Use this to suppress names
\nc{\mref}[1]{\ref{#1}}  % Use this to suppress names
\nc{\mbibitem}[1]{\bibitem{#1}} % Use this to show number name
\nc{\tred}[1]{\textcolor{red}{#1}}
\nc{\tblue}[1]{\textcolor{blue}{#1}}
\nc{\tgreen}[1]{\textcolor{green}{#1}}
\nc{\tpurple}[1]{\textcolor{purple}{#1}}
\nc{\btred}[1]{\textcolor{red}{\bf #1}}
\nc{\btblue}[1]{\textcolor{blue}{\bf #1}}
\nc{\btgreen}[1]{\textcolor{green}{\bf #1}}
\nc{\btpurple}[1]{\textcolor{purple}{\bf #1}}

\nc{\li}[1]{\tblue{LG:#1}}
\nc{\cm}[1]{\tblue{CB:#1}}
\nc{\nancy}[1]{\tred{NG:#1}}

\nc{\bin}[2]{ (_{\stackrel{\scs{#1}}{\scs{#2}}})}  %binomial coeff
\nc{\binc}[2]{ \left (\!\! \begin{array}{c} \scs{#1}\\
    \scs{#2} \end{array}\!\! \right )}  %binomial coeff
\nc{\bincc}[2]{  \left ( {\scs{#1} \atop
    \vspace{-1cm}\scs{#2}} \right )}  %binomial coeff
\nc{\egf}[2]{E_{#1, #2}}
\nc{\sarray}[2]{#1 \ \Big |\  #2 }
\nc{\ep}{\varepsilon}
\nc{\disp}[1]{\displaystyle{#1}}
\nc{\la}{\longrightarrow}
\nc{\rar}{\rightarrow}
\nc{\dar}{\downarrow}
\nc{\bs}{\bar{S}}
\nc{\dap}[1]{\downarrow \rlap{$\scriptstyle{#1}$}}
\nc{\dfpair}[2]{\left[ {{#1}\atop{#2}}\right]}
\nc{\uap}[1]{\uparrow \rlap{$\scriptstyle{#1}$}}
\nc{\defeq}{\stackrel{\rm def}{=}}
\nc{\dis}[1]{\displaystyle{#1}}
\nc{\dotcup}{\ \displaystyle{\bigcup^\bullet}\ }
\nc{\hcm}{\ \hat{,}\ }
\nc{\hts}{\hat{\shpr}}
\nc{\lts}{\stackrel{\leftarrow}{\shpr}}
\nc{\rts}{\stackrel{\rightarrow}{\shpr}}
\nc{\hcirc}{\hat{\circ}}
\nc{\catpr}{\diamond_l}
\nc{\rcatpr}{\diamond_r}
\nc{\lapr}{\diamond_a}
\nc{\lepr}{\diamond_e}
\nc{\ot}{\otimes}
\nc{\ora}[1]{\stackrel{#1}{\rar}}
\nc{\ola}[1]{\stackrel{#1}{\la}}%${\Bbb Z}$
\nc{\scs}[1]{\scriptstyle{#1}}
\nc{\sss}{\subsubsection}
\nc{\mrm}[1]{{\rm #1}}
\nc{\margin}[1]{\marginpar{\rm #1}}   %{\rm #1}}
\nc{\dirlim}{\displaystyle{\lim_{\longrightarrow}}\,}
\nc{\invlim}{\displaystyle{\lim_{\longleftarrow}}\,}
\nc{\msh}{\ast}
\nc{\mvp}{\vspace{0.3cm}}
\nc{\svp}{\vspace{2cm}}
\nc{\vp}{\vspace{8cm}}
\nc{\nsvs}{\vspace{-.5cm}}
\nc{\proofbegin}{\noindent{\bf Proof: }}
\nc{\proofend}{$\blacksquare$ \vspace{0.3cm}}
\nc{\sha}{\mbox{ \cyr X}}  %used to be \cyr
\font\cyrs=wncyr5
\nc{\ssha}{\mbox{\cyrs X}}

\nc{\shap}{\,\sha\,}
\nc{\csha}{\widehat{\sha}}
\nc{\ncsha}{\mbox{\cyr X}^{NC}}
\nc{\ncshao}{\mbox{\cyr X}^{NC,0}}
\nc{\lc}{\lfloor}
\nc{\rc}{\rfloor}
\nc{\ox}{\bar{\frakx}}
\nc{\shpr}{\diamond}    %Shuffle product
\nc{\shprc}{\shpr_c}
\nc{\spr}{\cdot}
\nc{\vep}{\varepsilon}
\nc{\mnote}[1]{{Note: {#1}}}
\nc{\shs}{\hspace{1cm}}
\nc{\onetree}{{\, \bullet \, }}
\nc{\frakTx}{\mathcal F}
\nc{\sqmon}[1]{\langle #1\rangle}
\nc{\free}[1]{\bar{#1}}
\nc{\oT}{\overline{\mathfrak T}}
\nc{\oF}{\overline{F}}

\nc{\smskip}{\smallskip}
\nc{\nsmskip}{\vspace{-.2cm}}
\nc{\nmedskip}{\vspace{-.5cm}}
\nc{\nbigskip}{\vspace{-1cm}}

\nc{\Aut}{\mrm{Aut}}
\nc{\BCH}{\mrm{BCH}}
\nc{\CK}{\mrm{CK}}
\nc{\mchar}{\mrm{Char}}
\nc{\End}{\mrm{End}}
\nc{\Gal}{\mrm{Gal}}
\nc{\GL}{\mrm{GL}}
\nc{\Hom}{\mrm{Hom}}
\nc{\id}{\mrm{id}}
\nc{\im}{\mrm{im}}
\nc{\incl}{\mrm{incl}}
\nc{\LR}{\mrm{LR}}
\nc{\RB}{\mrm{RB}}
\nc{\MS}{\mrm{ms}}
\nc{\OS}{\mrm{os}}
\nc{\SL}{\mrm{SL}}
\nc{\Spec}{\mrm{Spec}}
\nc{\Tr}{\mrm{Tr}}
\nc{\mtr}{\mrm{tr}}
\nc{\rba}{\mrm{RBA}}

\nc{\Alg}{\mathbf{Alg}}
\nc{\Ass}{\mathbf{Ass}}
\nc{\Bax}{\mathbf{Bax}}
\nc{\Coalg}{\mathbf{Coalg}}
\nc{\Dend}{\mathbf{DD}}
\nc{\DT}{\mathbf{DT}}
\nc{\Dias}{\mathbf{Dias}}
\nc{\bfk}{{\bf k}}
\nc{\bfone}{{\bf 1}}
\nc{\base}[1]{{e_{#1}}}
\nc{\detail}{\marginpar{\bf More detail}
    \noindent{\bf Need more detail!}
    \svp}
\nc{\Diff}{\mathbf{Diff}}
\nc{\Fil}{\mathrm{Fil}}
\nc{\gap}{\marginpar{\bf Incomplete}\noindent{\bf Incomplete!!}
    \svp}
\nc{\Leib}{\mathbf{Leib}}
\nc{\Lie}{\mathbf{Lie}}
\nc{\remarks}{\noindent{\bf Remarks: }}
\nc{\Rings}{\mathbf{Rings}}
\nc{\Sets}{\mathbf{Sets}}
\nc{\Vect}{\mathbf{Vec}}

\nc{\bfA}{{\mathbf A}}
\nc{\bfB}{{\mathbf B}}
\nc{\bfC}{{\mathbf C}}
\nc{\bfD}{{\mathbf D}}
\nc{\bfE}{{\mathbf E}}
\nc{\bfF}{{\mathbf F}}
\nc{\bfG}{{\mathbf G}}
\nc{\bfH}{{\mathbf H}}
\nc{\bfI}{{\mathbf I}}
\nc{\bfJ}{{\mathbf J}}
\nc{\bfK}{{\mathbf K}}
\nc{\bfR}{{\mathbf R}}

\nc{\BA}{{\Bbb A}}
\nc{\CC}{{\Bbb C}}
\nc{\DD}{{\Bbb D}}
\nc{\EE}{{\Bbb E}}
\nc{\FF}{{\Bbb F}}
\nc{\GG}{{\Bbb G}}
\nc{\HH}{{\Bbb H}}
\nc{\LL}{{\Bbb L}}
\nc{\NN}{{\Bbb N}}
\nc{\PP}{{\Bbb P}}
\nc{\QQ}{{\Bbb Q}}
\nc{\RR}{{\Bbb R}}
\nc{\TT}{{\Bbb T}}
\nc{\VV}{{\Bbb V}}
\nc{\ZZ}{{\Bbb Z}}

\nc{\cala}{{\mathcal A}}
\nc{\calb}{{\mathcal B}}
\nc{\calc}{{\mathcal C}}
\nc{\cald}{\mathcal{D}}
\nc{\cale}{{\mathcal E}}
\nc{\calf}{{\mathcal F}}
\nc{\calg}{{\mathcal G}}
\nc{\calh}{{\mathcal H}}
\nc{\cali}{{\mathcal I}}
\nc{\call}{{\mathcal L}}
\nc{\calm}{{\mathcal M}}
\nc{\caln}{{\mathcal N}}
\nc{\calo}{{\mathcal O}}
\nc{\calp}{{\mathcal P}}
\nc{\calr}{{\mathcal R}}
\nc{\calt}{{\mathcal T}}
\nc{\calu}{{\mathcal U}}
\nc{\calw}{{\mathcal W}}
\nc{\calx}{{\mathcal X}}
\nc{\CA}{\mathcal{A}}

\nc{\fraka}{{\frak a}}
\nc{\frakA}{{\frak A}}
\nc{\frakb}{{\frak b}}
\nc{\frakB}{{\frak B}}
\nc{\frakF}{{\frak F}}
\nc{\frakg}{{\frak g}}
\nc{\frakm}{{\frak m}}
\nc{\frakM}{{\frak M}}
\nc{\frakp}{{\frak p}}
\nc{\frakS}{{\frak S}}
\nc{\frakT}{{\frak T}}
\nc{\frakx}{{\mathbf x}}
\nc{\frakX}{{\mathfrak X}}

\font\cyr=wncyr10 %wncyr10

\begin{document}
\allowdisplaybreaks

\title[Generating functions and Rota--Baxter algebras]{Generating functions from the viewpoint of Rota-Baxter algebras}

\author{Nancy Shanshan Gu}
\address{Center for Combinatorics, LPMC-TJKLC, Nankai
University, Tianjin 300071, P. R. China}
\email{gu@nankai.edu.cn}

\author{Li Guo}
\address{Department of Mathematics and Computer Science,
         Rutgers University,
         Newark, NJ 07102, USA}
\email{liguo@rutgers.edu}

\date{\today}

%\begin{document}
\maketitle

%======================================================================================
\begin{abstract} We study generating functions in the context of Rota-Baxter algebras. We show that exponential generating functions can be naturally viewed
in a very special case of complete free commutative Rota-Baxter algebras. This allows us to use free Rota-Baxter algebras to give a broad class of algebraic structures in which generalizations of generating functions can be studied.
We generalize the product formula and composition formula for exponential power series. We also give generating functions both for known number families such as Stirling numbers of the second kind and partition numbers, and for new number families such as those from not necessarily disjoint partitions and partitions of multisets.
\end{abstract}

%MS numbers: 
%05A15   Exact enumeration problems, generating functions
%05A18   Partitions of sets
%05E40 Combinatorial aspects of commutative algebra
%11B73   Bell and Stirling numbers
%13J10 Complete rings, completion
%13N99 Differential algebra: None of the above, but in this section

\tableofcontents
\maketitle

\setcounter{section}{0}

\section{Introduction}

The importance of generating functions is well-known~\mcite{St,Wi}. We view generating functions in the framework of Rota-Baxter algebra whose importance in combinatorics was already noted by Rota~\mcite{Ro1,Ro4}. This viewpoint leads to a broad generalization of exponential generating functions.

There are many types of generating functions, such as the power series (or ordinary) generating functions $f(z)=\sum\limits_{n\geq 0} a_n z^n$, exponential generating functions $f(z)=\sum\limits_{n\geq 0} a_n \frac{z^n}{n!}$ and multi-variable generating functions. For some number sequences $(a_n)_{n\geq 0}$, their exponential generating functions have good expressions that are not available if power series generating functions are used. Such examples include the Bernoulli numbers and Bell numbers. Further, multi-variable generating functions encode multi-indexed number families that cannot be represented by one variable generating functions.

One way to view the power series generating functions
and exponential generating functions in the same framework is that
they each give a way to encode a number sequence $a_n, n\geq 0$ as
the coefficients of a linear combination with respect to a basis
of the power series algebra $\RR[[z]]$. The basis is
$z^n, n\geq 0$ for power series generating functions and
is the divided powers $z^n/n!, n\geq 0$ for exponential generating functions. The distinction
is the two different bases for the same power series ring. But it is also beneficial to view the difference externally: $v_k=z^k, k\geq 0$ is the standard
basis of the algebra
$$ \cala:= \prod_{k\geq 0} \RR\,v_k,\  \mbox{\ with componentwise product\ }v_m v_n =v_{m+n},$$
while $w_k=z^k/k!$ is the standard basis of the divided power algebra
$$ \calb:=\prod_{k\geq 0} \RR\, w_k, \mbox{\ with componentwise product\ } w_m w_n = \binc{m+n}{m} w_{m+n}.$$

We can take this point of view further and consider the following
general framework for generating functions:
A complete filtered $\RR$-algebra is a $\RR$-algebra $A$ with ideals $A_n,
n\geq 0$, such that $A_mA_n\subseteq A_{m+n}$ and
$A$ is complete with respect to the metric on $A$ induced by
$A_n$. In other words, the natural map
$$A\to \invlim A/A_n$$
is bijective. Let $\calu:=\{u_j, j\in J\},$ be a basis of $A$ that is compatible
with its filtration in the sense that $\calu \cap A_k$ is a basis
of $A_k$, $k\geq 0$.
Then a {\bf $\calu$-generating function} of a family of numbers
$a_j\in \RR, j\in J$, is the element
$\sum_{j\in J} a_j u_j$
in $A$. In this context, a power series generating function is
a $\calu$-generating function when $\calu$ is taken to be
the basis $v_k=\{z^k, k\geq 0\}$ in the complete filtered algebra
$$ \RR[[z]]= \prod_{k\geq 0} \RR\, z^k \cong \prod_{k\geq 0} \RR\, v_k =\cala$$
and an exponential generating function is a $\calu$-generating function when $\calu$ is taken
to be the basis $\{w_k=z^k/k!, k\geq 0\}$ in the complete
filtered algebra
$$ \RR[[z]]=\prod_{k\geq 0} \RR\, \frac{z^k}{k!}\cong \prod_{k\geq 0} \RR\, w_k=\calb.$$
In both cases, the complete filtration on $\RR[[z]]$ is given
by the ideals $z^n\RR[[z]]$ which is also $u_k \cala$ (resp.
$v_k \calb$).

Of course such a formal definition in such generality is of little use unless
\begin{enumerate}
\item
it can be naturally related to the ordinary generating functions or exponential generating functions;
\item
it is useful in the study of number sequences and number families.
\end{enumerate}

We will show that free Rota-Baxter algebras do give a generalization that satisfies these conditions.
In Section~\mref{sec:rba} we review the construction of free
commutative Rota-Baxter algebras and show that their completions
give a large class of complete filtered algebras.
In Section~\mref{sec:gen} we show that the simplest case of such complete free Rota-Baxter
algebras gives the exponential generating functions. We then focus on a ``twisted" variation of this case, namely on the $\lambda$-exponential generating functions. We generalize the classical Product Formula and Composition Formula for exponential generating functions to the twisted case. As applications, we give generating functions for Stirling numbers and for enumerations of not necessarily disjoint partitions. We consider another instance of complete free Rota-Baxter algebras in Section~\mref{sec:more}. By using $q$-series and applying the stuffle technique from the study of multiple zeta values in number theory, we show that elements of these algebras provide interesting generating functions for both number sequences such as factorials of square numbers, triangular numbers, pentagonal numbers, and multi-indexed families of numbers such as those from partitions of multisets.

\section{Free commutative complete Rota-Baxter algebras}
\mlabel{sec:rba}
In this section, we provide the algebraic framework for our extension of exponential generating functions. Let $\lambda\in \RR$ be a constant. A {\bf Rota--Baxter algebra} of weight $\lambda$ is a pair $(R,P)$ where $R$ is a unitary $\bfk$-algebra and $P:R\to R$ is a linear operator such that
\begin{equation}
P(x)P(y)=P(xP(y))+P(P(x)y)+\lambda P(xy),
\mlabel{eq:rbe}
\end{equation}
for any $x,\,y\in R$. Often $\theta=-\lambda$ is used, especially in the physics literature.
Let $(R,P)$ and $(R',P')$ be Rota-Baxter algebras of weight $\lambda$. A {\bf Rota-Baxter algebra homomorphism} $f$ from $(R,P)$ to $(R',P')$ is an algebra homomorphism
$f:R\to R'$ such that $f\circ P = P' \circ f$.

The study of Rota-Baxter algebras was started by G. Baxter~\mcite{Ba} in 1960 and was popularized largely by Rota and his school in the 1960s and 70s~\mcite{Ro1,Ro3} and again in 1990s~\mcite{Ro4}. In the recently years, there have been several interesting developments of Rota--Baxter algebras in
theoretical physics and mathematics, including quantum field theory, Yang--Baxter
equations, shuffle products, operads, Hopf algebras,
combinatorics and number theory~\cite{Bai,C-K1,EGK3,Gu2,G-Z}. See~\mcite{Gub} for further details.

\subsection{Free commutative Rota-Baxter algebras}
We recall the construction of free commutative Rota-Baxter algebras in terms of mixable shuffles.
See~\mcite{G-K1,G-K2,Gu3} for details.

Given a commutative algebra $\bfk$ which will often be taken as $\RR$, $\lambda\in \bfk$, and a $\bfk$-algebra $A$,
the free commutative Rota-Baxter $\bfk$-algebra on $A$ is defined to be a Rota-Baxter $\bfk$-algebra
$(\sha_{\bfk,\lambda}(A),P_A)$ together with a $\bfk$-algebra homomorphism
$j_A:A\to \sha_{\bfk,\lambda}(A)$ with the property that, for any Rota-Baxter $\bfk$-algebra
$(R,P)$ and any $\bfk$-algebra homomorphism $f:A\to R$, there is a unique Rota-Baxter
$\bfk$-algebra homomorphism $\tilde{f}:(\sha_{\bfk,\lambda}(A),P_A)\to (R,P)$ such that
$j_A\circ \tilde{f}=f$ as $\bfk$-algebra homomorphisms.

One realization of this free commutative Rota-Baxter algebra is given by the mixable shuffle Rota-Baxter algebra.
The mixable shuffle Rota-Baxter algebra is a pair $(\sha_{\bfk,\lambda}(A),P_A)$, where
$\sha_{\bfk,\lambda}(A)$ is a $\bfk$-algebra in which

\begin{itemize}
\item
the $\bfk$-module structure is given by
the direct sum
$$\bigoplus_{n=1}^\infty A^{\otimes n},\ {\rm\ where\ }
A^{\otimes n}=\underbrace{A\otimes_\bfk\ldots \otimes_\bfk A}_{n-{\rm factors}};$$
\item
the multiplication is given by the augmented mixable shuffle product $\diamond$, recursively defined on
$A^{\otimes m} \otimes A^{\otimes n}$ by
\begin{eqnarray*}
a_0 \diamond (b_0\otimes b_1 \otimes \ldots \otimes b_n)& =& a_0b_0 \otimes b_1\otimes \ldots
\otimes b_n, \\
(a_0\otimes a_1\otimes \ldots \otimes a_m)\diamond b_0
& =& a_0b_0 \otimes a_1 \otimes \ldots \otimes a_m,
\ a_i,\, b_j\in A^{\otimes 1}=A,
\end{eqnarray*}
 and
\begin{eqnarray}
\mlabel{eq:shuf}
\lefteqn{(a_0\otimes a_1\otimes \ldots \otimes a_m)\diamond (b_0\otimes b_1 \otimes \ldots
\otimes b_n) } \\
&=& (a_0 b_0) \otimes \big ((a_1\otimes \ldots \otimes a_m)\diamond
(1\otimes b_1\otimes \ldots \otimes b_n)\big ) \notag \\
&& +
(a_0 b_0) \otimes \big ((1\otimes a_1 \otimes \ldots \otimes a_m)\diamond (b_1\otimes \ldots
\otimes b_n)\big) \notag \\
&& + \lambda
a_0 b_0 \otimes \big ((a_1\otimes \ldots \otimes a_m)\diamond (b_1\otimes \ldots \otimes
b_n)\big),\ a_i,\ b_j\in A
\notag
\end{eqnarray}
\end{itemize}
with the convention that
$$ a\diamond (1\ot b)=a\ot b,\ (1\ot a)\diamond b = b\ot a,\
    a\diamond b = ab,\ \mbox{\ for\ } a, b\in A.$$
The Baxter operator $P_A$ is defined by
$$P_A(a_1\otimes \ldots \otimes a_m)=1\otimes a_1\otimes \ldots \otimes a_m,\
a_1\otimes \ldots \otimes a_m \in A^{\otimes m},\
m\geq 1.$$
Since the mixable shuffle product is compatible with the product on $A$, we will often
suppress the notation $\diamond$. We will also suppress $\bfk$ and $\lambda$ from
$\sha_{\bfk,\lambda}(A)$ when there is not danger of confusion.

Note that assuming $P_A$ is a Rota-Baxter operator and thus satisfies Eq.~(\mref{eq:rbe}), then Eq.~(\mref{eq:shuf}) follows.
For example, we have
\begin{eqnarray*}
\lefteqn{(a_0\otimes a_1\otimes a_2) (b_0\otimes b_1) }\\
&=&a_0b_0\otimes \left ( (a_1\otimes a_2)(1\otimes b_1)
+ b_1 (1\otimes a_1\otimes a_2)
+ \lambda (a_1\otimes a_2) b_1 \right ) \\
&=& a_0b_0\otimes \left ( a_1 \otimes (a_2 (1\otimes b_1)+(b_1 (1\otimes a_2))
    + \lambda a_2 b_1)
+ b_1 \otimes a_1\otimes a_2
+ \lambda a_1 b_1 \otimes a_2 \right ) \\
&=& a_0b_0 \otimes \left ( a_1 \otimes a_2 \otimes b_1 + a_1 \otimes
b_1 \otimes a_2
    + \lambda a_1 \otimes a_2 b_1 + b_1 \otimes a_1 \otimes a_2 + \lambda a_1 b_1 \otimes a_2
    \right ).
\end{eqnarray*}

\begin{theorem} {\bf(}\cite[Theorem 4.1]{G-K1}{\bf)}
\label{thm:shua}
For any $\bfk$-algebra $A$,
$(\sha(A),P_A)$, together with the natural embedding
$j_A:A\rightarrow \sha(A)$,
is a free Baxter $\bfk$-algebra on $A$ (of weight $\lambda$) in the sense that
the triple $(\sha(A),P_A,j_A)$ satisfies the following
universal property:
For any Baxter $\bfk$-algebra $(R,P)$ and any
$\bfk$-algebra map
$\varphi:A\rar R$, there exists
a unique Baxter $\bfk$-algebra homomorphism
$\tilde{\varphi}:(\sha(A),P_A)\rar (R,P)$ such that
the  diagram
\[\xymatrix{
A \ar[rr]^(0.4){j_A} \ar[drr]_{\varphi}
    && \sha(A) \ar[d]^{\tilde{\varphi}} \\
&& R } \]
commutes.
\end{theorem}

Alternatively, $\sha(A)$ can defined to be
the tensor product algebra $A\ot \sha^+(A)$ where the multiplication on $\sha^+(A)=\bigoplus_{k\geq 1} A^{\ot k}$ is
given in the explicit form by the mixable shuffle product~\mcite{G-K1} and in the recursive form by
a generalization~\mcite{E-G1,G-Z} of the {\bf quasi-shuffle algebra} defined by Hoffman~\mcite{Ho2} in the study of multiple zeta values. It is also shown that $\sha^+(A)$ is the free commutative tridendriform algebra~\mcite{Lo}.

Quasi-shuffle product is also known as harmonic product~\mcite{Ho1}
and coincides with the stuffle product~\mcite{3BL,Br} in the study
of multiple zeta values. Variations of the stuffle product have also
appeared in ~\mcite{Ca2,Eh}. It is shown~\mcite{E-G1} to be the same
as the mixable shuffle product~\mcite{G-K1,G-K2} which is also
called overlapping shuffles~\mcite{Ha} and generalized
shuffles~\mcite{Go3}, and can be interpreted in terms of Delannoy
paths~\mcite{A-H,Fa,Lo}.

\subsection{Complete free Rota-Baxter algebras}

For a given commutative algebra $A$, it is easy to see that
the submodules
$$\Fil^k \sha(A):= \bigoplus_{i \geq k} A^{\ot i},\ i\geq 0$$
of $\sha(A)$ are ideals of $\sha(A)$. They are in fact Rota-Baxter
ideals of $\sha(A)$ in the sense that $P_A( \Fil^k \sha(A))\subseteq \Fil^k \sha(A)$. Further, $\cap_{k\geq 0} \Fil^k \sha(A)=0.$
Thus $$\csha(A):= \invlim \sha(A)/\Fil^k \sha(A) \cong \prod_{k\geq 0} A^{\ot k}$$
is a complete filtered algebra and contains $\sha(A)$ as a
subalgebra. It coincides with the complete
free commutative Rota-Baxter algebra defined in~\mcite{G-K2}.

\section{Generating functions from Rota-Baxter algebras}
\mlabel{sec:gen}
We first interpret exponential generating functions in terms of the special case of free commutative Rota-Baxter algebra when the ring $A$ is the base field $\RR$ and when the weight $\lambda\in \RR$ is zero. We then consider the nonzero weight case (still taking $A=\RR$) and generalize to this case the Product Formula and Composition Formula for exponential generating functions. Applications of these formulas are provide. Other instances of free Rota-Baxter algebras will be investigated in the next section. 

\subsection{Connection with exponential power series}

Let $A=\RR$. Then
\begin{equation}
\sha_\lambda(A)=\sha_\lambda(\RR)=\oplus_{k\geq 1} \RR^{\ot k}
= \oplus_{k\geq 1} \RR\, \bfone_k,
\mlabel{eq:freer0}
\end{equation}
where $\bfone_k=\underbrace{1\ot \cdots \ot 1}_{(k+1)-{\rm terms}}$.
The augmented shuffle product in this special case is
\begin{equation}
 \bfone_m \diamond \bfone_n= \sum_{k=0}^{\min(m,n)} \lambda^k \binc{m+n-k}{m}\binc{m}{k} \bfone_{m+n-k}.
\mlabel{eq:freer}
\end{equation}

When $\lambda=0$, we have
$$ \bfone_m \diamond \bfone_n= \binc{m+n}{m} \bfone_{m+n},$$
giving the divided powers. We then have
$\csha_\lambda(\RR)=\prod_{k\geq 0} \RR \bfone_{k}$.
This is also the cofree differential algebra~\mcite{Ke}
with the differential operator $d(x_n)=x_{n-1}, d(1)=0$.

Denote $x_n=x^n/n!$ (divided powers). Then as an algebra,
$$ \RR[x]=\oplus_{n\geq 0} \RR x_n$$
with multiplication given by $x_m x_n =\binc{m+n}{m} x_{m+n}$.
This extends to an isomorphism
$$\RR[[x]] \to \widehat{\sha}(\RR)=\prod_{k\geq 0} \RR\, \bfone_k,
\quad x_k\mapsto \bfone_k,\ k\geq 0.$$

Through this isomorphism the theory of exponential generating function
is translated to a theory of generating functions in $\widehat{\sha}(\RR)$, and can be generalized
to $\widehat{\sha}_\lambda(A)$ for other algebras $A$ and other weight $\lambda$.
We will next set up the foundation of our approach by proving the product formula and composition formula for $\lambda$-exponential generating functions.

\subsection{Product formula for $\lambda$-exponential generating functions}

We first consider free Rota-Baxter algebra $\sha(\RR)$ on $\RR$ of weight 1. It is given by Eq.~(\mref{eq:freer0}) with product given
by Eq.~(\mref{eq:freer}) with $\lambda =1$.

We quote from~\mcite{St} the following simple yet fundamental property of exponential generating functions which underlies the prominent role played by
these generating functions. See~\mcite{St} for details.
For any function $f:\NN\to \RR$, let
$$E_f(x)=\sum_{k\geq 0} f(n) \frac{x^n}{n!}$$
be the exponential generating function of $f$.
\begin{proposition} {\rm (Stanley~\cite[Proposition 5.1.1]{St})}
Let $\#Y$ be the cardinality of a finite set $Y$. Given functions
$f,g: \NN\to \RR$, define a new function $h:\NN\to \RR$ by the rule
\begin{equation}
 h(\#X)=\sum_{(S,T)} f(\#S)g(\#T),
 \mlabel{eq:cup}
\end{equation}
where $X$ is a finite set, and where $(S,T)$ ranges over all weak ordered partitions of $X$ into two blocks, i.e., $S\cap T=\emptyset$ and $S\cup T=X$. Then
\begin{equation}
E_h(x)=E_f(x)E_g(x).
\mlabel{eq:exp}
\end{equation}
\mlabel{thm:exp}
\end{proposition}

\begin{definition}
For $f:\NN\to \RR$, we call \begin{equation}
\egf{\lambda}{f}:=\sum_{k\geq 0} f(k) \bfone_k\in \csha_\lambda(\RR)
\mlabel{eq:legf}
\end{equation}
the {\bf $\lambda$-exponential generating function} (or {\bf  $\lambda$-EGF} in short) of $f$. When there is no danger of confusion, we will suppress $\lambda$ from the notation.
\mlabel{de:legf}
\end{definition}
When $\lambda=0$, this recovers the exponential generating function.
We prove the following generalization of Proposition~\mref{thm:exp}.

\begin{proposition}
Given functions $f,g: \NN\to \RR$, define a new function $h:\NN\to \RR$ by the rule
\begin{equation} h(\#X)=\sum_{(S,T)\in \calp(X)^2, S\cup T=X} \lambda^{\#(S\cap T)} f(\#S)g(\#T),
\mlabel{eq:cup1}
\end{equation}
where $X$ is a finite set and $\calp(X)$ is the power set of $X$. Thus in the sum $(S,T)$ ranges over all ordered pairs of subsets of $X$ such that $S\cup T=X$ (but not necessarily $S\cap T=\emptyset$). Then
\begin{equation}
\egf{\lambda}{h}=\egf{\lambda}{f}\, \egf{\lambda}{g}
\in \csha_\lambda(\RR).
\mlabel{eq:exp1}
\end{equation}
\mlabel{pp:exp1}
\end{proposition}
\begin{proof}
We have
\begin{eqnarray*}
\egf{\lambda}{f}\, \egf{\lambda}{g}&=&
\big(\sum_{m\geq 0} f(m)\bfone_m\big) \big(\sum_{n\geq 0} g(n)\bfone_n\big) \\
&=& \sum_{m,n\geq 0} f(m)g(n) \bfone_m \bfone_n\\
&=& \sum_{m,n\geq 0} f(m)g(n) \sum_{i=0}^{\min(m,n)} \lambda^i \binc{m+n-i}{m}\binc{m}{i} \bfone_{m+n-i}
\end{eqnarray*}
by Eq.~(\mref{eq:freer}). Note that $
\binc{m+n-i}{m}\binc{m}{i}=\binc{m+n-i}{i,m-i,n-i}.$ So setting
$u=m+n-i$, $u_1=i, u_2=m-i, u_3=n-i$, we have $u=u_1+u_2+u_3$ and
\begin{equation*}
\egf{\lambda}{f}\, \egf{\lambda}{g}= \sum_{u=0}^\infty \Big(
\sum_{\tiny\begin{array}{c}(u_1,u_2,u_3)\in \NN^3\\ u_1+u_2+u_3=u
\end{array}} \lambda^{u_1}\binc{u}{u_1,u_2,u_3} f(u_1+u_2)g(u_1+u_3)
\Big) \bfone_u.
\end{equation*}
Now the theorem follows since, for given $(u_1,u_2,u_3)\in \NN^3$,
$\binc{u}{u_1,u_2,u_3}$ is the number of ways of partitioning a size $u$ set into three subsets of size $u_1,u_2$ and $u_3$. It is also the number of ways of taking subsets $S$ of size $u_1+u_2$ and $T$ of size $u_1+u_3$ of a set $X$ of size $u_1+u_2+u_3$ such that $S\cup T=X$ and $\#(S\cap T)=u_1$.
\end{proof}

{}From the proof of Proposition~\mref{pp:exp1} we have the following explicit formula for $h(\#X)$ in Eq.~(\mref{eq:cup1}).
\begin{corollary}
The numbers in Eq.~(\mref{eq:cup1}) is given by

\begin{equation}
h(u)=\sum_{\tiny\begin{array}{c}(u_1,u_2,u_3)\in \NN^3\\ u_1+u_2+u_3=u \end{array}} \lambda^{u_1}\binc{u}{u_1,u_2,u_3}
f(u_1+u_2)g(u_1+u_3), u\geq 0.
\mlabel{eq:exph}
\end{equation}
\mlabel{co:exp1}
\end{corollary}

In analogy to \mcite{St}, we have the following combinatorial significance of Proposition~\mref{pp:exp1}. Suppose we have two types of structures, say $\alpha$ and $\beta$, which can be put on a finite set $X$. We assume that the allowed structures depend only on the cardinality of $X$. A new ``combined" type of structure, denoted $\alpha\cup_\lambda \beta$, can be put on $X$ by placing structures of type $\alpha$ and $\beta$ on subsets $S$ and $T$, respectively, of $X$ such that $S\cup T=X$ (but not necessarily $S\cap T=\emptyset$). If $f(k)$ and $g(k)$ denote the sums of allowed structures on a $k$-set of type $\alpha$ and $\beta$, respectively, then the right-hand side of Eq.~(\mref{eq:cup1}) counts the sum of the allowed structures of type $\alpha \cup_\lambda \beta$ with a ``correction factor" $\lambda^{\# (S\cap T)}$ that measures the overlap of $S$ and $T$. In particular, when $\lambda=0$, then we recover the classical case in~\mcite{St}.

\begin{example}
As an application, suppose a group $X$ of $n$ children participates in an event where they can play two games $\alpha$ and $\beta$. Of course a child can play either one or both or none of the games. Let the group of children who play game $\alpha$ and $\beta$ be $S$ and $T$, respectively. Suppose game $\alpha$ is a simple hit-or-missing game (such as shorting a ball). So there are $2^{\# S}$ possible outcomes. Suppose game $\beta$ is a competition that results a linear order of the participants. So there are $(\# T)!$ possible outcomes. Let $h(n)=h(\# X)$ be the possible outcomes of the event. Let $\egf{\lambda}{f}=\sum_{k\geq 0} f(k) \bfone_k$ and $\egf{\lambda}{g}=\sum_{k\geq 0} g(k) \bfone_k$ be the $\lambda$-EGF for the outcomes of game $\alpha$ and $\beta$, respectively. Then the $\lambda$-EGF for the whole event is
$$ \egf{\lambda}{h}=\egf{\lambda}{f}\egf{\lambda}{g}.$$
\end{example}

Using an inductive argument, we obtain the following generalization of Proposition 5.1.3 in~\mcite{St}.

\begin{theorem}
$(${\rm Generalized Product Formula}$)$ Fix $k\in \PP$ with $k\geq
2$ and functions $f_1,\cdots,f_k:\NN\to \RR$. Define a new function
$h:\NN\to \RR$ by
\begin{equation}
 h(\# T)=\sum_{(T_1,\cdots,T_k)} \lambda^{\#T_1+\cdots+\#T_k-\#(T_1\cup \cdots \cup T_k)} f_1(\#T_1)\cdots f_k(\#T_k),
\mlabel{eq:cupm1}
\end{equation}
or
\begin{equation}
 h(\# T)=\sum_{(T_1,\cdots,T_k)} \lambda^{\eta (T_1,\cdots,T_k)} f_1(\#T_1)\cdots f_k(\#T_k),
\mlabel{eq:cupm2}
\end{equation}
where $(T_1,\cdots,T_k)$ ranges over all weakly ordered (not necessarily disjoint) subsets $T_1,\cdots, T_k$ of $S$ such that $T_1\cup \cdots \cup T_k=T$
and
$$\eta(T_1,\cdots,T_k)=\sum_{\tiny \begin{array}{c} I\subseteq [k] \\ \#I\geq 2\end{array}} (-1)^{\# I}
\# \big(\bigcap_{i\in I} T_i\big).$$
Then
\begin{equation}
\egf{\lambda}{h}= \prod_{i=1}^k \egf{\lambda}{f_i}
\mlabel{eq:expm}
\end{equation}
in $\csha_\lambda(\RR)$.
\mlabel{thm:expm}
\end{theorem}
\begin{proof}
We first use induction on $k\geq 2$ to prove that the fuction
$h:\NN\to \RR$ defined by Eq.~(\mref{eq:cupm1} satisfies
Eq.~(\mref{eq:expm}). Since $\#(T_1\cap T_2)=\#T_1 +\#T_2-\#(T_1\cup
T_2)$, the case of $k=2$ is proved in Proposition~\mref{pp:exp1}.
Assume that the claim has been proved for $k=n\geq 2$ and consider
the case of $k=n+1$. Let $f_1,\cdots,f_{n+1}:\NN \to \RR$ be given.
Let $f_{1n}:\NN\to \RR$ be defined such that
$$\egf{\lambda}{f_{1n}}=\egf{\lambda}{f_1}\cdots \egf{\lambda}{f_n}.$$
Then by the induction hypothesis,
\begin{equation}
f_{1n}(\#T)=\sum_{(T_1,\cdots,T_n)} \lambda^{\#T_1+\cdots \#T_n-\#(T_1\cup \cdots \cup T_n)} f_1(\# T_1)\cdots f_n (\# T_n).
\mlabel{eq:f1}
\end{equation}
Since $\egf{\lambda}{f} =\egf{\lambda}{f_{1n}}\egf{\lambda}{f_{n+1}}$
by definition, by Proposition~\mref{pp:exp1} (that is, the case when $k=2$) we have
\begin{equation} f(\# X)=\sum_{\tiny{
\begin{array}{c} (X_1,X_2)\in \calp(X)^2 \\ X_1\cup
X_2=X\end{array}}} \lambda^{\#X_1+ \# X_2-\#(X_1\cup X_2)}
f_{1n}(\#X_1) f_{n+1}(\# X_2). \mlabel{eq:f2}
\end{equation}
Combining Eq.~(\mref{eq:f1}) and Eq.~(\mref{eq:f2}) we have
{\small
\begin{eqnarray*}
&& f(\# X)\\
&=& \hspace{-1.2cm}\sum_{\tiny{ \begin{array}{c} (X_1,X_2)\in \calp(X)^2 \\ X_1\cup X_2=X\end{array}}}
\hspace{-1cm} \lambda^{\#X_1+ \#X_2-\#(X_1\cup X_2)} \left ( \sum_{\tiny \begin{array}{c}(X_{1,1},\cdots, X_{1,n})\in \calp(X_1)^n\\ X_{1,1}\cup \cdots \cup X_{1,n}=X_1\end{array}}
\hspace{-1.5cm}\lambda^{\sum_{i=1}^n\#X_{1,i}-\#(\cup_{i=1}^n X_{1,i})}
\prod_{i=1}^n f_i(\#X_{1,i})\right ) f_{n+1}(\#X_2)
\\
&=&
\hspace{-1cm}\sum_{\tiny \begin{array}{c} (X_{1,1},\cdots, X_{1,n},X_2) \in \calp(X)^{n+1} \\ X_{1,1}\cup\cdots\cup X_{1,n}\cup X_{n+1} =X\end{array} }
\hspace{-1.5cm} \lambda^{\#X_1+ \#X_2-\#(X_1\cup X_2)+ \sum_{i=1}^n\#X_{1,i}-\#(\cup_{i=1}^n X_{1,i})}
f_1(\# X_{1,1}) \cdots f_n(\# X_{1,n}) f_{n+1}(\# X_{2})
\end{eqnarray*}
}
since any $(X_{1,1},\cdots, X_{1,n},X_2) \in \calp(X)^{n+1}$ such that $X_{1,1}\cup\cdots\cup X_{1,n}\cup X_{n+1} =X$ corresponds uniquely to a $(X_1,X_2)\in \calp(X)^2$ with $X_1\cup X_2=X$ together with a $(X_{1,1},\cdots,X_{1,n})\in \calp(X_1)^n$ with
$X_{1,1}\cup \cdots \cup X_{1,n} =X_1$. Thus to complete the induction we only need to show
\begin{align*}
&\#X_1+ \#X_2-\#(X_1\cup X_2)+ \sum_{i=1}^n\#X_{1,i}-\#(\cup_{i=1}^n X_{1,i}) \\
= &\#X_{1,1}+\cdots \#X_{1,n}+\#X_2 - \#(X_{1,1}\cup \cdots \cup X_{1,n} \cup X_2)
\end{align*}
which is clear since $X_1=X_{1,1}\cup \cdots \cup X_{1,n}$.

Now to finish proving the theorem, we just need to show that Eq.~(\mref{eq:cupm1}) agrees with Eq.~(\mref{eq:cupm2}), that is, $$ \#T_1+\cdots + \#T_k- \#(T_1\cup \cdots \cup T_k) =
\eta(T_1,\cdots, T_k).$$
This follows from the well-known Euler characteristic for the
subsets $T_1,\cdots,T_k$:
$$ \# \bigcup_{j=1}^k T_i = \sum_{\emptyset \neq I\subseteq [k]} (-1)^{\# I -1} \# (\cap_{i\in I} T_i)
=\#T_1+\cdots + \#T_k - \eta (T_1,\cdots,T_k).
$$
\end{proof}

For the differential operator $d_{A}$, we have $d_{A}(\bfone_n)=\bfone_{n-1}$,
$d_{A}(1)=0$, and the following result.
\begin{equation*}
d_{A}(E_{\lambda, f})=d_{A}(\sum_{n\geq 0}f(n)\bfone_n)=\sum_{n\geq 0}f(n+1)\bfone_n.
\end{equation*}
As the applications of Theorem \ref{thm:expm}, we have the following results.
\begin{corollary}
Let $S$ be a finite set. Given functions $f,g: \NN \to \RR$, define
new functions $h_1,\ h_2,\ h_3$, and $h_4$ on $\NN$ as follows:
\begin{eqnarray*}
h_1(\#S) &=& f(\#S)+g(\#S)\\
h_2(\#S) &=& (\#S)f(\#T)+\lambda (\#S)f(\#S), \qquad \text{where}\ \#T=\#S-1\\
h_3(\#S) &=& f(\#T) , \qquad \text{where}\ \#T=\#S+1\\
h_4(\#S) &=& (\#S)f(\#S)+\lambda(\#S)f(\#T), \qquad \text{where}\ \#T=\#S+1.
\end{eqnarray*}
Then
\begin{eqnarray*}
E_{\lambda, h_1} &=& E_{\lambda, f}+E_{\lambda, g}\\
E_{\lambda, h_2} &=& \bfone_1 \diamond E_{\lambda, f}\\
E_{\lambda, h_3} &=& d_A(E_{\lambda, f})\\
E_{\lambda, h_4} &=& \bfone_1 \diamond d_A(E_{\lambda, f}).
\end{eqnarray*}
\end{corollary}
The proof is easy and will be omitted.

\begin{corollary}$($\mcite{Gu3}$)$ For $\lambda \in \RR$, we have the generating function
\begin{equation*}
\frac{1}{1-(1 \otimes 1)u}=\sum_{n=0}^{\infty} \sum_{k=0}^nk!S(n,k)
\lambda^{n-k}u^n \bfone_k=\sum_{n=0}^{\infty} \sum_{k=0}^{\infty}n!S(n+k,n)\lambda^k u^{n+k} \bfone_n,
\end{equation*}
where $S(n,k), n, k\geq 0,$ are the Stirling numbers of the second kind.
\end{corollary}
\begin{proof}
Since $\frac{1}{1-(1 \otimes 1)u}=\sum_{n=0}^{\infty}(1 \otimes 1)^nu^n$, we set $E_{\lambda,h}=\sum_{k=0}^{\infty}h(k)\bfone_k=(1 \otimes 1)^n$ and $E_{\lambda,f}=\sum_{m=0}^{\infty}f(m)\bfone_m
=1 \otimes 1$. Hence, we have
\begin{equation*}
E_{\lambda,h}=E_{\lambda,f}^n=\underbrace{E_{\lambda,f} \diamond
E_{\lambda,f} \diamond \cdots \diamond E_{\lambda,f}}_n.
\end{equation*}
For $E_{\lambda,f}=1 \otimes 1$, we have
\begin{equation*}
f(m)=\left\{\begin{array}{cc}1,&\quad m=1,\\0,&\quad m\neq 1.\end{array}\right.
\end{equation*}
According to Theorem \ref{thm:expm}, we get
\begin{equation}
 h(\# T)=\sum_{(T_1,\cdots,T_n)} \lambda^{\#T_1+\cdots+\#T_n-\#(T_1\cup \cdots \cup T_n)} f(\#T_1)\cdots f(\#T_n),
\label{prod}
\end{equation}
where $(T_1,\cdots,T_n)$ ranges over all weak ordered (not necessarily disjoint) subsets $T_1,\cdots, T_n$ of $T$ such that $T_1\cup \cdots \cup T_n=T$.

Notice that the nonzero terms in the right hand side of \eqref{prod} are those which satisfy the condition $\#T_1=\ldots=\#T_n=1$. Therefore,
we can reformulate \eqref{prod} as follows.
\begin{equation*}
h(k)=\sum_{S_1,\ldots,S_k}\lambda^{n-k},
\end{equation*}
where $(S_1,\cdots,S_k)$ ranges over all the ordered partitions of $[n]$ such that $S_i \cap S_j=\emptyset$ for $1 \leq i,j \leq k$ and $S_1\cup \cdots \cup S_k=[n]$.
That is to say,
\begin{equation*}
h(k)=k!S(n,k)\lambda^{n-k}.
\end{equation*}
Hence, we get
\begin{equation}\label{1otimes1}
(1 \otimes 1)^n=\sum_{k=0}^n k!S(n,k)\lambda^{n-k} \bfone_k.
\end{equation}
Then the corollary follows.
\end{proof}

From the above proof, we have the following generation function.
\begin{corollary}For $\lambda \in \RR$, we have the generating function
\begin{equation*}
e^{(1 \otimes 1)u}=\sum_{n=0}^{\infty} \sum_{k=0}^nk!S(n,k)
\lambda^{n-k}u^n \bfone_k/n!=\sum_{n=0}^{\infty} \sum_{k=0}^{\infty}n!S(n+k,n)
\lambda^ku^{n+k} \bfone_n/(n+k)!.
\end{equation*}
\end{corollary}

Setting
\begin{equation*}
E_{\lambda,I}:=\sum_{n=0}^{\infty}\bfone_n,
\end{equation*}
we consider the expression for $E_{\lambda,I}^k$.
\begin{corollary}
For $\lambda=1$, we have
\begin{equation*}
E_{1,I}^k=\sum_{n=0}^{\infty}(2^k-1)^n \bfone_n.
\end{equation*}
\end{corollary}
\begin{proof}
By Theorem \ref{thm:expm}, we have
\begin{equation}
E_{1,I}^k = (\sum_{n=0}^{\infty}\bfone_n)^k = \sum_{n=0}^{\infty}\sum_{\tiny (T_1,\cdots,T_k)}\bfone_n,
\mlabel{eq:e1i}
\end{equation}
where $(T_1,\cdots,T_k)$ ranges over all weakly ordered (not necessarily disjoint) subsets $T_1,\cdots, T_k$ of $[n]$ such that $T_1\cup \cdots \cup T_k=[n]$.
Let
$$f_i:[n]\to \{0, 1\}, \quad f_i(j)=\left \{\begin{array}{ll} 1, & j\in T_i, \\ 0, & j\neq T_i, \end{array} \right .  1\leq i\leq k,
$$
be the characteristic function of $T_i, 1\leq i\leq k.$ Then
$$S_j:=\{i\in [k]\,|\, f_i(j)\neq 0\},  1\leq j\leq n,$$
are nonempty subsets of $[k]$. Conversely, given nonempty subsets
$S_j, 1 \leq  j\leq n,$ define
$$g_j:[k]\to \{0, 1\}, \quad g_j(i)=\left\{\begin{array}{ll} 1, & i\in S_j, \\ 0, & i\not\in S_j\end{array} \right . 1\leq j\leq n.$$
Then
$$T_i:= \{ j\in [n]\,|\, g_j(i)\neq 0\}, \quad 1\leq i\leq k,$$
form a weakly ordered (not necessarily disjoint) subsets of $[n]$ such that $T_1\cup \cdots \cup T_k=[n].$ Thus the set of such weakly ordered subsets $\{T_1,\cdots,T_k\}$ is in bijection with the set of nonempty subsets $\{S_1,\cdots,S_n\}$ of $[k]$.
Therefore in Eq.~(\mref{eq:e1i}), we have
$$\sum_{\tiny (T_1,\cdots,T_k)} 1 = (2^k-1)^n,$$
hence the corollary.
\end{proof}

Now we give a notation that will be used in the next corollary.
\begin{definition}\label{B-cover}
For $k, \ell \geq 1$ and $\ell \leq n \leq k\ell$, let $B(n,k,\ell)$ denote the number of $k$-tuples $(T_1,\cdots,T_k)$
of (not necessarily disjoint) subsets $T_1,\cdots, T_k$ of $[n]$
such that $T_1 \cup \cdots \cup T_k=[n]$ and $\#T_i=\ell$ for $1
\leq i \leq k$. For other values of $n, k, \ell$, we set $B(n,k,\ell)=0$.
\end{definition}
\begin{corollary}\label{g-Bnkl}For $\ell \geq 1$, we have the generating function
\begin{equation*}
\frac{\bfone_\ell}{1-\bfone_\ell}=\sum_{k=1}^{\infty}(\bfone_\ell)^k=\sum_{k=1}^{\infty}\sum_{n=l}^{k\ell}\lambda^{k\ell-n}B(n,k,\ell)\bfone_n.
\end{equation*}
\end{corollary}
\begin{proof}
Set
\begin{equation*}
E_{\lambda,f}:=\bfone_\ell=\sum_{m=0}^{\infty}f(m)\bfone_m,
\quad
\text{ where } f(m)=\delta_{\ell,m}.
\end{equation*}
Then by Theorem \ref{thm:expm}, we have
\begin{align}
(\bfone_\ell)^k &= \sum_{n=0}^{\infty}\sum_{\tiny (T_1,\cdots,T_k)}\lambda^{\#T_1+\cdots+\#T_k-\#(T_1 \cup \cdots \cup T_k)}
f(\#T_1)\cdots f(\#T_k)\bfone_n \nonumber\\
&=\sum_{n=l}^{\ell k}\lambda^{k\ell-n}B(n,k,\ell)\bfone_n,\label{g-n-k-l}
\end{align}
giving the corollary.
\end{proof}
In Eq.~\eqref{g-n-k-l}, when $\ell=1$, we get Eq.~\eqref{1otimes1}; when $\lambda=0$, since $B(k\ell, k,\ell)=\frac{(k\ell)!}{(\ell!)^k}$, we obtain
\begin{equation*}
(1 \otimes 1^{\otimes \ell})^k=\frac{(k\ell)!}{(\ell!)^k}(1 \otimes 1^{\otimes k\ell}).
\end{equation*}
So we get the following result.
\begin{equation}\label{bfone_kl}
\frac{1}{1-\bfone_\ell}=\sum_{k=0}^{\infty}(\bfone_\ell)^k=\sum_{k=0}^{\infty}\frac{(k\ell)!}{(\ell!)^k}\bfone_{k\ell}.
\end{equation}

\subsection{Composition formula for $\lambda$-exponential generating functions}

We now give a generalization of the Composition Formula~\mcite{St}.
\begin{theorem}
{\rm (Composition Formula)~\cite[Theorem 5.1.4]{St}} Given functions
$f:\PP \to \RR$ and $g: \NN \to \RR$ with $g(0)=1$, define a new
function $h:\NN\to \RR$ by
\begin{eqnarray}
h(\# S) &=& \sum_{\pi=\{B_1,\cdots,B_k\}\in \Pi (S)}
f(\#B_1) \cdots f(\#B_k) g(k), \#S>0,
\mlabel{eq:part1}
\\
h(0)&=& 1,
\end{eqnarray}
where the sum ranges over all partitions $\pi=\{B_1,\cdots,B_k\}$ of the finite set $S$. Then
\begin{equation}
E_h(x)=E_g(E_f(x)).
\mlabel{eq:comp1}
\end{equation}
\mlabel{thm:comp1}
\end{theorem}

Now we give some definitions before we prove the main theorem for the composition rule in $\csha_\lambda(\RR)$.

\begin{definition}\label{generalized-par}
A collection $\{B_1,\cdots,B_k\}$ of $[n]$ is called a {\bf
generalized partition (with distinct max) of $[n]$} if
\begin{enumerate}
\item $B_i \neq \emptyset$, $1\leq i\leq k$;
\item $B_1 \cup B_2 \cup\cdots \cup B_k=[n]$;
\item $\max B_1<\max B_2<\cdots <\max B_k.$
\end{enumerate}
Let $\Pi'([n])$ denote the set of generalized partitions of $[n]$.
Let $\overline{S}(n,k)$ denote the number of the generalized partitions of $[n]$ with $k$ blocks. By convention,
we put
$$\overline{S}(0,0)=1, \overline{S}(n,0)=0 \text{
for } n \geq 1, \text{ and } \overline{S}(n,k)=0 \text{ for }
k>n\geq 1.$$

Let $\overline{B}(n)$ denote the number of the generalized partitions
of $[n]$. That means $\overline{B}(n)=\sum_{k=1}^n\overline{S}(n,k)$.
\end{definition}

For example, $\overline{S}(3,2)=8$. All the generalized partitions of $[3]$ with $2$ blocks are listed below.
\begin{equation*}
\begin{array}{llll}
\{\{1\},\{2,3\}\},&\{\{1\},\{1,2,3\}\},&\{\{1,2\},\{3\}\},&\{\{1,2\},\{2,3\}\},\\
\{\{2\},\{1,3\}\},&\{\{2\},\{1,2,3\}\},&\{\{1,2\},\{1,3\}\},&\{\{1,2\},\{1,2,3\}\}.
\end{array}
\end{equation*}
But $\{\{1,3\},\{2,3\}\}$ and $\{\{1,3\},\{1,2,3\}\}$ are not generalized partitions of $[3]$.

With the extra restriction that $B_i\cap B_j=\emptyset$ for $i \neq j$, $\{B_1, B_2, \cdots, B_k\}$ is a partition of $S$ \cite[p. 33]{St}. Recall that the number of such partitions is $S(n,k)$, the Stirling number of the second kind, and the number of all partitions equals to $B(n)$, the $n$-th Bell number. For this reason, we also call $\overline{S}(n,k)$ and $\overline{B}(n)$ the generalized Stirling numbers of the second kind and the generalized Bell numbers, respectively.

Similar to the recursive formula for Stirling numbers of the second kind, we have the following formula for $\overline{S}(n,k)$.

\begin{proposition}\label{p-recurrence} For $n,k \geq 1$, we have
\begin{equation}\label{recurrence}
\overline{S}(n,k)=\sum_{i=0}^{n-1}{n-1 \choose i}2^i\overline{S}(i,k-1).
\end{equation}
$$\overline{S}(n,1)=1,\qquad \overline{S}(n,n)=2^{n \choose 2}.$$
\end{proposition}

\begin{proof}
We first prove the recurrence.

Every generalized partition of $[n]$ with $k$ blocks is determined uniquely in the following three independent steps:
\begin{enumerate}
\item
choose a subset $X$ of $[n-1]$ with $i$ elements, $0\leq i\leq n-1$;
\item
choose a generalized partition $\{B_1,\cdots,B_{k-1}\}$ of $X$;
\item
choose a subset $B_k$ of $[n]$ that contains $([n-1]\backslash X)\cup \{n\}$ and contains a subset of $X$.
\end{enumerate}
The number of choices of the three steps are ${n-1 \choose i}$,
$\overline{S}(i,k-1)$ and $2^{i}$, respectively, hence the recursion.

Since there is only one way to put $[n]$ in one block, we have
$\overline{S}(n,1)=1$.

If $\{B_1,\cdots,B_n\}$ is a generalized partition of $[n]$ into $n$ blocks, then we must have $\max B_i=i, 1\leq i\leq n$. Then $B_i=\{i\}\cup B'_i$ for a subset $B'_i$ of $[i-1]$ (with the convention that $[0]=\emptyset$. Thus there are $2^{i-1}$ choices for $B_i, 1\leq i\leq n$. Therefore, we have
$$S(n,n)=2^02^1\cdots2^{n-1}=2^{n \choose 2}.$$
It also follows from the recursion.
\end{proof}

\begin{proposition}\label{S'nk} For $n\geq k \geq 1$, we have
\begin{equation*}
\overline{S}(n,k)=2^{k \choose 2}\hspace{-1cm}\sum_{\tiny \begin{array}{c}m_1, m_2, \cdots, m_{n-k}\\
1 \leq m_1<m_2<\cdots<m_{n-k}\leq n-1\end{array}}
\prod_{i=1}^{n-k}(2^{k-m_i+i}-1).
\end{equation*}\end{proposition}

\begin{proof}
For a generalized partition $\{B_1,\cdots,B_k\}$ of $[n]$ with $k$ blocks. We have $\max B_1<\cdots <\max B_k$. For the convenience of the proof, we use the notation $B'_{\max B_i}:=B_i, 1\leq i\leq k$. In order words, $B'_j$ denotes the block whose largest element is $j$. Such $j$'s form a subset $N$ of $[n]$ with $k$ elements. Denote
$$ [n]\backslash N = \{1\leq m_1<m_2<\cdots<m_{n-k}\leq n-1\}.$$

With these notations, there are three steps in constructing a generalized partition of $[n]$ with $k$ blocks.

\begin{enumerate}
 \item[Step 1.] Choose a subset $N$ of $[n]$ of cardinality $k$ to be the largest elements of the blocks. Equivalently, choose a subset $\{1\leq m_1<\cdots,m_{n-k}\leq n-1\}$ to be the complement of $N$.
\item[Step 2.] Determine the contribution of $N$ in the generalized partition with $k$ blocks.
\item[Step 3.] Determine the contribution of  $[n]\backslash N$ in the generalized partition with $k$ blocks.
\end{enumerate}

We next derive formulas for Step 2 and Step 3.

\noindent
{\bf Step 2.}
We construct a generalized partition of $N$ with $k$ blocks.
We know that all the $k$ elements in $N$ are the largest elements of the blocks $B'_j, j\in N$. We can also view $N$ as the result after removing the $n-k$ elements of $[n]\backslash N$ from $[n]$. Viewing each elements of $[n]\backslash N$ as a cutting point, the elements of $N\subseteq [n]$ are cut into
$n-k+1$ segments even though some of the segments might be empty. For example, taking $n=9$ and $N=\{2, 3, 7,  8, 9\}$. Then $[9]\backslash N=\{m_1=1,m_2=4,m_3=5,m_4=6\}$, cutting $[9]$ into five groups some of which might be empty:
\begin{eqnarray*}
 &&\{ \}; \\
 &&\{2, 3\};\\
 && \{ \}; \\
 && \{ \};\\
 && \{7, 8, 9\}.
 \end{eqnarray*}

By collecting the segments together this way, we can put the $N$ elements and their corresponding $B'_j$\,s into groups as follows.

\begin{equation*}
\begin{array}{llll}
B'_1,& B'_2, &\cdots, &B'_{m_1-1};\\
B'_{m_1+1},&B'_{m_1+2},&\cdots,&B'_{m_2-1};\\
B'_{m_2+1},&B'_{m_2+2},&\cdots,&B'_{m_3-1};\\
&&\vdots&\\
B'_{m_{n-k-1}+1}, & B'_{m_{n-k-1}+1},&\cdots,&B'_{m_{n-k}-1};\\
B'_{m_{n-k}+1}, &B'_{m_{n-k}+2}, &\cdots, &B'_{n}.
\end{array}
\end{equation*}
We note that $B'_{m_1}, B'_{m_2}, \cdots, B'_{m_{n-k}}$ are not defined and don't appear among the above groups. Consider $B'_j$ with $1\leq j \leq m_1-1$, namely the $B'_j$ in the first line. Since $B'_j$ can contain $j$ and any subset of the $j-1$ elements in $N$ that are less than $j$, there are $2^{j-1}$ choices for $B'_j$. Therefore, the total choices of all the blocks in the first
line are $2^02^1\cdots 2^{m_1-2}$. Next consider $B'_j$ with $m_1+1\leq j \leq m_2-1$. Then $B'_j$ contains $j$ and any subset of the elements of $N$ that are less than $j$. There are $j-2$ such elements $\{1,\cdots, m_1-1, m_1+1,\cdots,j-1\}$ since $m_1$ is not in $N$. Thus there are $2^{j-2}$ choices for $B'_j$. In general, each $B'_j$ on the $i$-th line has $j-i$ choices, where $1\leq i\leq n-k+1$. Therefore, the number of
generalized partitions of $N$ with $k$ blocks is.

\begin{align*}
\overline{S}(N,k)&=(2^02^1\cdots 2^{m_1-2})\times (2^{m_1-1}2^{m_1}\cdots2^{m_2-3})\times \cdots \times (2^{m_{n-k}-n+k}2^{m_{n-k}-n+k+1}\cdots 2^{k-1})\\
&=2^{k \choose 2}.
\end{align*}

\noindent
{\bf Step 3:}
Let a generalized partition of $N$ with $k$ blocks be given. We next determine the number of ways to put the elements $m_1,m_2,\cdots,m_{n-k}$ into the blocks to form a generalized partition of $[n]$ with $k$ blocks.

By the definition of $B'_j$, we see that $m_i, 1\leq i\leq n-k$ can not be put in any $B'_j$ with $j<m_i$ and hence can only be put in the $B'_j$'s after the $i$-th line in the above table. Since the $B'_j$ have union $[n]$ and can have overlaps, every $m_i, 1\leq i \leq n-k$ can appear in any nonzero number of the blocks after the $i$-th line in the above table. Hence there are $2^{k-(m_i-i)}-1$ possible ways to add $m_i$ to the existing blocks. Therefore, the total number of choices of adding $[n]\backslash N$ to any given generalized partition of $N$ with $k$ blocks is
$$(2^{k-(m_1-1)}-1)\times (2^{k-(m_2-2)}-1)\times\cdots\times (2^{k-(m_{n-k}-(n-k))}-1)=\prod_{i=1}^{n-k}(2^{k-m_i+i}-1).$$
\smallskip

Putting the three steps together, we have
\begin{align*}
\overline{S}(n,k)&=\sum_{\tiny \begin{array}{c}m_1, m_2, \cdots, m_{n-k}\\
1 \leq m_1<m_2<\cdots<m_{n-k}\leq n-1\end{array}}2^{k \choose 2}
\prod_{i=1}^{n-k}(2^{k-m_i+i}-1).
\end{align*}
We complete the proof.
\end{proof}

We have the following special cases of Proposition \ref{S'nk}.
\begin{corollary}\label{S'2n-1} For $n \geq 1$, we have
$$\overline{S}(n,2)=3^{n-1}-1, \qquad \overline{S}(n,n-1)=2^{n-1 \choose 2}(2^n-n-1).$$
\end{corollary}

\begin{proof}
We give a direct proof of the first equation it can also be obtained from Proposition
\ref{p-recurrence} and Proposition \ref{S'nk} by setting $k=2$.

Let $\{B_1,B_2\}$ be a generalized partitions of $[n]$ with $2$ blocks. Then $1\leq \max B_1\leq n-1$ and $\max B_2=n$. If $\max B_1=t, 1\leq t\leq n-1$ and $\# B_1=j, 1\leq j\leq t$, then there are ${t-1 \choose j-1}$ choices of $B_1$. For each such a choice of $B_1$, there are $2^j$ choices of $B_2$ since $B_2$ must contain $\{t+1,\cdots,n\}\cup ([t]\backslash B_1)$ but can contain any subset of $B_1$. Thus for given $1\leq t\leq n-1$, there are
$\sum_{j=1}^t {t-1 \choose j-1} 2^j=2 (3^{t-1})$ choices. Summing over $t$, we have
$$\overline{S}(n,2)=\sum_{t=1}^{n-1}2\cdot 3^{t-1}=3^{n-1}-1.$$

For $\overline{S}(n,n-1)$, setting $k=n-1$ in Proposition \ref{S'nk}, we have
\begin{align*}
\overline{S}(n,n-1)&=2^{n-1 \choose 2}\sum_{m_1=1}^{n-1}(2^{n-1-m_1+1}-1)\\
&=2^{n-1 \choose 2}\sum_{m_1=1}^{n-1}(2^{n-m_1}-1)\\
&=2^{n-1 \choose 2}(2^n-n-1).
\end{align*}

\end{proof}

To obtain the generalization of the Composition Formula, we first need to give a proper definition for the composition of two elements in $\csha_\lambda(\RR)$.
We do this by a suitable generalization of the construction in~\mcite{K-P} which treated the case of $\lambda=0$.

Set
\begin{equation*}
E_{\lambda,f}=\sum_{n=0}^{\infty}f(n)\bfone_n.
\end{equation*}
For the differential operator $d_A$, we have $d_A(\bfone_n)=\bfone_{n-1}$, and $d_A(1)=0$. Then, we have
\begin{equation*}
d_A(E_{\lambda,f})=\sum_{n=0}^{\infty}f(n+1)\bfone_n.
\end{equation*}
Next, for the Rota-Baxter operator $P_A$, we have $P_A(\bfone_n)=\bfone_{n+1}$. Then we have
\begin{equation*}
P_A(E_{\lambda,f})=\sum_{n=1}^{\infty}f(n-1)\bfone_n.
\end{equation*}
Therefore, we have
\begin{equation*}
d_A(P_A(E_{\lambda,f}))=E_{\lambda,f}.
\end{equation*}

\begin{definition}
We define the {\bf $n$-th divided power} $E_{\lambda,f^{[n]}}$ of $E_{\lambda,f}$ recursively by $E_{\lambda,f^{[0]}}=1$ and
$$E_{\lambda, f^{[n]}}=P_A\left(E_{\lambda, f^{[n-1]}}\diamond d_A(E_{\lambda, f})\right), \quad n\geq 1.$$
\end{definition}
Thus for any $n \geq 1$, we have
$d_A(E_{\lambda, f^{[n]}})=E_{\lambda, f^{[n-1]}}\diamond d_A(E_{\lambda, f})$.

\begin{definition}\label{composition}
Given functions $f: \mathbb{P} \rightarrow \mathbb{R}$ and $g:\mathbb{N} \rightarrow \mathbb{R}$, we define the {\bf composition} of $E_{\lambda,g}$ and $E_{\lambda,f}$ by
$E_{\lambda,g}(E_{\lambda,f})=\sum_{k=0}^{\infty}g(k)E_{\lambda,f^{[k]}}$.
\end{definition}
First, we give a formula related to $E_{\lambda,f^{[k]}}$.
\begin{lemma}\label{E-f-k}
Given a function $f:\mathbb{P} \rightarrow \mathbb{R}$, for $k \geq 1$, we have
\begin{equation}\label{Ef-k}
E_{\lambda,f^{[k]}}=\sum_{n=k}^{\infty}\sum_{\{B_1,\cdots,B_k\}\in \Pi'([n])}
\lambda^{\#B_1+\cdots+\#B_k-\#(B_1\cup \cdots \cup B_k)}f(\#B_1)\cdots f(\#B_k)\bfone_n,
\end{equation}
where $\{B_1, B_2\cdots,B_k\}$ ranges over all the generalized partitions of $[n]$.
\end{lemma}
\begin{proof}
When $k=1$, according to the definition of $E_{\lambda,f^{[k]}}$, we have
\begin{equation*}
E_{\lambda,f^{[1]}} = P_A\left(E_{\lambda,f^{[0]}} \diamond d_A(E_{\lambda,f})\right)= P_A\left(1 \diamond \sum_{n=1}^{\infty}f(n)\bfone_{n-1}\right)
= P_A\left(\sum_{n=1}^{\infty}f(n)\bfone_{n-1}\right)= \sum_{n=1}^{\infty}f(n)\bfone_n.
\end{equation*}
Setting $k=1$ in the right hand side of Eq.~\eqref{Ef-k}, we have
\begin{equation*}
\sum_{n=1}^{\infty}\sum_{\{B_1\}\in \Pi'([n])}\lambda^{\#B_1-\#B_1}f(\#B_1)\bfone_n=\sum_{n=1}^{\infty}f(n)\bfone_n.
\end{equation*}
Therefore, Eq.~\eqref{Ef-k} holds for $k=1$.
Assume that Eq.~\eqref{Ef-k} holds for $k$. Then for $k+1$, we have
{\small \begin{align}
E_{\lambda,f^{[k+1]}} &= P_A\left(E_{\lambda,f^{[k]}} \diamond d_A(E_{\lambda,f})\right) \nonumber\\
&=P_A\left(\sum_{n=k}^{\infty}\sum_{\{B_1,\cdots,B_k\} \in \Pi'([n])}
\lambda^{\#B_1+\cdots+\#B_k-\#(B_1\cup \cdots \cup B_k)}f(\#B_1)\cdots f(\#B_k)\bfone_n \right.  \nonumber\\[-10pt]
&\qquad\qquad\qquad\qquad\qquad\qquad\qquad\qquad\qquad\left. \diamond\ \left(\sum_{m=1}^{\infty}f(m)\bfone_{m-1}\right)\right) \nonumber\\
&=P_A\left(\sum_{n=0}^{\infty}\sum_{m=1}^{\infty}\sum_{\{B_1,\cdots,B_k\}\in \Pi'([n])}\lambda^{\#B_1+\cdots+\#B_k-\#(B_1\cup \cdots \cup B_k)}f(\#B_1)\cdots f(\#B_k)f(m)\bfone_n\diamond \bfone_{m-1}\right) \nonumber\\
&=P_A\left(\sum_{n=0}^{\infty}\sum_{m=1}^{\infty}\sum_{\{B_1,\cdots,B_k\}\in \Pi'([n])}\lambda^{\#B_1+\cdots+\#B_k-\#(B_1\cup \cdots \cup B_k)}f(\#B_1)\cdots f(\#B_k)f(m)\right. \nonumber\\
&\qquad\qquad\qquad\qquad\qquad\qquad\qquad \left.\cdot\sum_{i=0}^{\min(n,m-1)}\lambda^i{m+n-i-1 \choose n}{n \choose i}\bfone_{m+n-i-1}\right) \nonumber\\
&=P_A\left(\sum_{u=0}^{\infty}\sum_{\tiny\begin{array}{c}(u_1,u_2,u_3)\in \NN^3\\ u_1+u_2+u_3=u \end{array}} \lambda^{u_1}\binc{u}{u_1,u_2,u_3}\right. \nonumber\\
&\ \left.\cdot \Big(\sum_{\{B_1,\cdots,B_k\}\in \Pi'([u_1+u_2])}
\lambda^{\#B_1+\cdots+\#B_k-\#(B_1\cup \cdots \cup B_k)}f(\#B_1)\cdots f(\#B_k)f(u_1+u_3+1)\Big)\bfone_u\right) \nonumber\\
&=\sum_{u=1}^{\infty}\sum_{\tiny\begin{array}{c}(u_1,u_2,u_3)\in \NN^3\\ u_1+u_2+u_3=u-1 \end{array}} \lambda^{u_1}\binc{u-1}{u_1,u_2,u_3} \nonumber\\
&\ \cdot \Big(\sum_{\{B_1,\cdots,B_k\}\in \Pi'([u_1+u_2])}
\lambda^{\#B_1+\cdots+\#B_k-\#(B_1\cup \cdots \cup B_k)}f(\#B_1)\cdots f(\#B_k)f(u_1+u_3+1) \Big) \bfone_u.\nonumber
\end{align}}
As shown at the end of the proof of Proposition~\mref{pp:exp1}, we have
$$
\sum_{\tiny\begin{array}{c}(u_1,u_2,u_3)\in \NN^3\\ u_1+u_2+u_3=u-1 \end{array}} \lambda^{u_1}\binc{u-1}{u_1,u_2,u_3}=
\sum_{\tiny \begin{array}{c}\{B,B'_{k+1}\} \in \calp([u-1])^2\\B \cup B'_{k+1}=[u-1]\end{array}}\lambda^{\#B+\#B'_{k+1}-\#(B\cup B'_{k+1})}.$$
Thus we get
\begin{align*}
E_{\lambda,f^{[k+1]}}&=\sum_{u=1}^{\infty}\sum_{\tiny \begin{array}{c}\{B,B'_{k+1}\} \in \calp([u-1])^2\\B \cup B'_{k+1}=[u-1]\end{array}}\lambda^{\#B+\#B'_{k+1}-\#(B\cup B'_{k+1})}\\
&\ \cdot \Big(\sum_{\{B_1,\cdots,B_k\}\in \Pi'(B)}
\lambda^{\#B_1+\cdots+\#B_k-\#(B_1\cup \cdots \cup B_k)}f(\#B_1)\cdots f(\#B_k)f(\#B'_{k+1}+1)\Big) \bfone_u.
\end{align*}
Denote $B_{k+1}:=B'_{k+1}\cup\{u\}$. Then $\{B,B_{k+1}\}$ forms a generalized partition of the set $[u]$. Further, if $\{B,B_{k+1}\}$ is a generalized partition of $[u]$ and $\{B_1,\cdots,B_k\}$ is
a generalized partition of $B$, then $\{B_1,\cdots,B_{k+1}\}$ is a generalized partition of $[u]$.
Therefore we have
\begin{align*}
E_{\lambda,f^{[k+1]}}
&=\sum_{u=1}^{\infty}\sum_{\{B,B_{k+1}\}\in \Pi'([u])}\sum_{\{B_1,\cdots,B_k\}\in \Pi'(B)}\lambda^{\#B+\#B_{k+1}-1-(\#(B\cup B_{k+1})-1)+\#B_1+\cdots+\#B_k-\#B}  \label{EF-k-2}\\
&\qquad\qquad\qquad\qquad\qquad\qquad\qquad\qquad\qquad \cdot f(\#B_1)\cdots f(\#B_k)f(\#B_{k+1})\bfone_u \nonumber\\
&=\sum_{u=1}^{\infty}\sum_{\{B_1,\cdots,B_{k+1}\}\in \Pi'([u])}
\lambda^{\#B_1+\cdots+\#B_{k+1}-\#(B_1\cup \cdots \cup B_{k+1})}f(\#B_1)\cdots f(\#B_{k+1})\bfone_u \nonumber\\
&=\sum_{u=k+1}^{\infty}\sum_{\{B_1,\cdots,B_{k+1}\}\in \Pi'([u])}
\lambda^{\#B_1+\cdots+\#B_{k+1}-\#(B_1\cup \cdots \cup B_{k+1})}f(\#B_1)\cdots f(\#B_{k+1})\bfone_u. \nonumber
\end{align*}
This completes the induction.
\end{proof}

\begin{theorem}\label{g-composition}
$(${\rm Generalized Composition Formula}$)$ Given functions $f: \mathbb{P} \rightarrow \mathbb{R}$ and $g:\mathbb{N} \rightarrow \mathbb{R}$ with $g(0)=1$,
define a new function $h: \mathbb{N} \rightarrow \mathbb{R}$ by
\begin{align*}
h(n)&=\sum_{\tiny{\begin{array}{c}\{B_1,\cdots,B_k\}\in \Pi'([n])\\ k\geq 1\end{array}}}
\lambda^{\#B_1+\cdots+\#B_k-\#(B_1\cup \cdots \cup B_k)}f(\#B_1)\cdots f(\#B_k)g(k),\quad n>0,\\
h(0)&=1,
\end{align*}
where the sum ranges over all generalized partitions $\{B_1,\cdots,B_k\}$ of $[n]$. Then
\begin{equation*}
E_{\lambda,h}=E_{\lambda,g}(E_{\lambda,f}).
\end{equation*}
$($Here $E_{\lambda,f}=\sum_{n=1}^{\infty}f(n)\bfone_n$, since $f$ is only defined on positive integers.$)$
\end{theorem}
Notice that Theorem \ref{g-composition} reduces to Theorem \ref{thm:comp1} when $\lambda=0$.

\begin{proof} By Definition~\ref{composition} and Lemma~\ref{E-f-k}, we have
\begin{align*}
 &E_{\lambda,g}(E_{\lambda,f})\\
&= \sum_{k=0}^{\infty}g(k)E_{\lambda,f^{[k]}}\\
&= g(0)E_{\lambda,f^{[0]}}+\sum_{k=1}^{\infty}g(k)E_{\lambda,f^{[k]}}\\
&=1+\sum_{k=1}^{\infty}g(k)\sum_{n=k}^{\infty}\sum_{\{B_1,\cdots,B_k\}\in \Pi'([n])}
\lambda^{\#B_1+\cdots+\#B_k-\#(B_1\cup \cdots \cup B_k)}f(\#B_1)\cdots f(\#B_k)\bfone_n\\
&=1+\sum_{n=1}^{\infty}\sum_{k=1}^n \sum_{\{B_1,\cdots,B_k\}\in \Pi'([n])}
\lambda^{\#B_1+\cdots+\#B_k-\#(B_1\cup \cdots \cup B_k)}f(\#B_1)\cdots f(\#B_k)g(k)\bfone_n,
\end{align*}
as needed.
\end{proof}

As applications of Theorem \ref{g-composition}, we give the following corollaries, providing further justification for the notations $\overline{S}(n,k)$ and $\overline{B}(n)$. Recall that $E_{\lambda,I}=\sum_{n=0}^{\infty}\bfone_n$.

\begin{corollary} Let $k \geq 0$,  $E_{\lambda,f}=E_{\lambda,I}-1$ and $E_{\lambda,g}=\bfone_k$. Then we have
\begin{equation*}
E_{\lambda,g}(E_{\lambda,f})=\left\{\begin{array}{cc}\displaystyle\sum_{n=0}^{\infty}S(n,k)\bfone_n, & \lambda=0,\\[20pt]
\displaystyle\sum_{n=0}^{\infty}\overline{S}(n,k)\bfone_n, & \lambda=1,\end{array}\right.
\end{equation*}
Especially, when $\lambda=1$ and $k=2$, we have
\begin{equation}\label{S2}
E_{1,g}(E_{1,f})=1+\sum_{n=1}^{\infty}(3^{n-1}-1)\bfone_n.
\end{equation}
When $\lambda =1$ and $k=n-1$, we have
\begin{equation}\label{Sn-1}
E_{1,g}(E_{1,f})=1+\sum_{n=1}^{\infty}2^{n-1 \choose 2}(2^n-n-1)\bfone_n.\\
\end{equation}
\mlabel{co:1k}
\end{corollary}

\begin{proof}
The first equation follows from Theorem~\mref{g-composition} since, for our choices, $$ f(0)=0,  f(i)=1, i\geq 1 \text{ and } g(i)=\delta_{k,i}, i\geq 0.$$

We then get Eqs.~\eqref{S2} and \eqref{Sn-1} by Corollary \ref{S'2n-1}.
\end{proof}

By a similar argument, we obtain
\begin{corollary} Let $E_{\lambda,f}=E_{\lambda,I}-1=\sum\limits_{n=1}^\infty \bfone_n$ and $E_{\lambda,g}=E_{\lambda,I}=\sum\limits_{n=0}^\infty\bfone_n$. Then we have
\begin{equation*}
E_{\lambda,g}(E_{\lambda,f})=\left\{\begin{array}{cc}1+\displaystyle\sum_{n=1}^{\infty}B(n)\bfone_n, & \lambda=0,\\[20pt]
1+\displaystyle\sum_{n=1}^{\infty}\overline{B}(n)\bfone_n, & \lambda=1,\end{array}\right.
\end{equation*}
where $B(n)$ is the $n$-th Bell number, and $\overline{B}(n)$ is defined in Definition \ref{generalized-par}.
\end{corollary}

Before proving the next result, we give a definition similar to Definition~\ref{B-cover}.

\begin{definition}
For $k, \ell \geq 1$ and $k+\ell-1 \leq n \leq k\ell$, let $B'(n,k,\ell)$ denote the number of $k$-tuples $(T_1,\cdots,T_k)$
of (not necessarily disjoint) subsets $T_1,\cdots, T_k$ of $[n]$
such that $T_1 \cup \cdots \cup T_k=[n]$, $\#T_i=\ell$ for $1
\leq i \leq k$, and all the largest elements of the subsets are different from each other. For the other values of $n, k, \ell$, we set $B'(n,k,\ell)=0$.
\end{definition}

Then by the same proof as for Corollary~\mref{co:1k}, we have
\begin{corollary} If $E_{\lambda,f}=\bfone_\ell$ and $E_{\lambda,g}=\bfone_k$ for $k,\ell \geq 1$, then we have
\begin{equation*}
E_{\lambda,g}(E_{\lambda,f})=1+\sum_{n=1}^{\infty}B'(n,k,\ell)\bfone_n.
\end{equation*}
\end{corollary}

\section{Generating functions in free commutative Rota-Baxter algebras}
\mlabel{sec:more}
In this section, we will study generating functions in $\csha_\lambda(x)$, the complete free commutative Rota-Baxter algebra on one generator $x$.
In Section~\mref{ss:wzg} and~\mref{ss:nzgen}, we will discuss the weight zero case and the
general weight case, respectively.
We first recall the notations:
$$ \sha_\lambda(x):=\sha_\lambda(\RR[x])=\bigoplus_{k=1}^\infty \RR[x]^{\ot k}= \RR\{ x^{n_0}\ot x^{n_1}\ot \cdots \ot x^{n_k}\,|\, n_0,\cdots, n_k\geq 0, k\geq 0\}.$$
Recall that the product in $\sha_\lambda(x)$ and hence in $\csha_\lambda(x)$ is given by the mixable shuffle product $\shpr$ which will be suppressed.

We state some definitions and notations which
are used in this section.
A composition of $n$ is an expression of $n$ as an ordered
sum of positive integers. If exactly $k$ summands appear in a
composition $\sigma$, we say that $\sigma$ has $k$ parts, and we
call $\sigma$ a $k$-composition.
Let $I=(i_1,\cdots,i_t)\in \mathbb{P}^t$. Define the {\bf norm} of
$I$ to be $|I|=\sum_{s=1}^t i_s$ and its {\bf length} to be
$\ell(I)=t$. Also denote $\binc{n}{I}=\binc{n}{i_1,\cdots,i_t}$.

\subsection{The weight zero case}
\mlabel{ss:wzg}
We consider the case when $\lambda=0$. Then for $x^{n_0}\ot \cdots \ot x^{n_k}$ and $x^{m_0}\ot \cdots \ot x^{m_\ell}$ we have
$$ (x^{n_0}\ot \cdots \ot x^{n_k})  (x^{m_0}\ot \cdots \ot x^{m_\ell}) = x^{n_0+m_0}\ot\big( (x^{n_1}\ot \cdots \ot x^{n_k})\ssha (x^{m_1}\ot \cdots \ot x^{m_\ell})\big),$$
where $\ssha$ is the shuffle product. In particular, $$(1\ot x^{\ot n}) (1\ot x^{\ot m})= {n+m \choose n}\, 1 \ot x^{\ot (n+m)}$$
and more generally,
\begin{equation} \prod_{i=1}^k (1\ot x^{\ot n_i}) = {n_1+\cdots+n_k \choose n_1,\cdots,n_k} \, (1\ot x^{\ot (n_1+\cdots +n_k)}).
\mlabel{eq:multsh}
\end{equation}
Also,
\begin{equation}
 \prod_{i=1}^k (1\ot x^i) = \sum_{\sigma\in S_k} 1\ot x^{\sigma(1)}\ot \cdots \ot x^{\sigma(k)},
\mlabel{eq:permsh}
\end{equation}
where $S_k$ is the set of permutations on $[k]$.

Then we easily obtain

\begin{proposition}
For $n,k \geq 1$, we have the following relations in $\sha_0(x)$.
\begin{align}
(1\ot x^{\ot k})^n &= \frac{(kn)!}{(k!)^n}(1\ot x^{\ot kn}),\label{lambda0-nk}\\
\prod_{k=1}^n(1\otimes x^{\otimes k}) &= {{n+1 \choose 2} \choose
{1,2,3,\ldots,n}}(1\otimes x^{\otimes {n+1 \choose 2}}).\nonumber
\end{align}
Further,
$$\prod_{n=1}^{\infty}(1+(1\otimes x^n)) = \sum_{n=0}^{\infty}\sum_{|J|=n}(1\otimes x^{\otimes J}),$$
where $J$ denotes all the compositions of $n$ with distinct parts.
\mlabel{thm:multsh}
\end{proposition}

\begin{proof}
The first two equations follow from Eq.~(\mref{eq:multsh}). The third equation follows from Eq.~(\mref{eq:permsh}).
\end{proof}

By treating $1 \ot x^{\ot k} \in \csha_0(X)$ as the base $q$ in basic hypergeometric series, we next apply $q$-series identities to derive generating functions for some special sequences.

Recall the $q$-shifted factorial~\mcite{G-R}
$$(a;q)_\infty=
\prod_{n=1}^{\infty}(1-aq^{n-1}) \text{\ \  and \ \  }(a;q)_n
=\frac{(a;q)_\infty}{(aq^n;q)_\infty}, \text{ for } n\in \mathbb{Z}.$$ \\
Then we have the following identities~\mcite{Be} in
$\QQ[[q]]$.
\begin{align}
\varphi(q)&:=\sum_{n=-\infty}^{\infty}q^{n^2} = (-q;q^2)_{\infty}^2(q^2;q^2)_{\infty},\mlabel{eq:phi}\\
\psi(q)&:=\sum_{n=0}^{\infty}q^{n+1 \choose 2} = \frac{(q^2;q^2)_{\infty}}{(q;q^2)_{\infty}}, \mlabel{eq:psi}\\
f(-q)&:=\sum_{n=-\infty}^{\infty}(-1)^nq^{\frac{n(3n-1)}{2}} = (q;q)_{\infty}.\mlabel{eq:f}
\end{align}
We have the following generating functions for the sequences of the factorial of square numbers, triangular numbers, and pentagonal numbers.

\begin{theorem} We have
\begin{align*}
&\sum_{n=-\infty}^{\infty}(n^2)!(1 \ot x^{\ot n^2}) = \prod_{n=1}^{\infty}(1+(2n-1)!(1 \ot x^{\ot (2n-1)}))^2(1-(2n)!(1 \ot x^{ \ot 2n})),\\
&\sum_{n=0}^{\infty}{n+1 \choose 2}!(1 \ot x^{\ot {n+1 \choose 2}}) = \prod_{n=1}^{\infty}\frac{(1-(2n)!(1 \ot x^{\ot (2n)}))}{(1-(2n-1)!(1 \ot x^{\ot (2n-1)}))},\\
&\sum_{n=-\infty}^{\infty}(-1)^n\left(\frac{n(3n-1)}{2}\right)!(1 \ot x^{\ot \frac{n(3n-1)}{2}}) = \prod_{n=1}^{\infty}(1-n!(1 \ot x^{\ot n})).
\end{align*}
\end{theorem}
\begin{proof}
Note that $\RR[[q]]$ is the free completed commutative $\RR$-algebra on $q$ in the sense that for any complete $\RR$-algebra $A$ and nonunit $a\in A$, there is unique algebra homomorphism $g:\RR[[q]]\to A$ such that $g(q)=a$. Thus about identities in $\RR[[q]]$ induce similar identities in $A$. Take $A=\csha(x)$, $a=1\ot x$ and let
$g_x: \RR[[q]]\to \csha(x)$ be the algebra homomorphism such that $g_x(q)=1\ot x$. Applying $g_x$ to the above three $q$-series identities in $ \RR[[q]]$, we obtain identities for $x\in \csha(x)$.

For example, applying $g_x$ to Eq.~(\mref{eq:phi}), we have
$$ \sum_{n=-\infty}^\infty (1\ot x)^{n^2}= \left(\prod_{n=1}^\infty (1+(1\ot x)^{\ot (2n-1)})\right)^2 \left( \prod_{n=1}^\infty (1-(1\ot x)^{2n})\right).$$
Then using the following special case of \eqref{lambda0-nk}
$$(1\ot x)^n = n!(1\ot x^{\ot n}),$$
we obtain the first equation.
The proofs of the other two equations are similar.
\end{proof}

\subsection{The general weight case}
\mlabel{ss:nzgen}

Before giving the main theorem in this section, we first need some definitions and notations.

Let $\pi(kn)$ denote the set of compositions of $kn$ with each part
less or equal to $n$. For a composition $I=(i_1, i_2,\cdots,i_t) \in
\pi(kn)$, we define a {\bf restricted partition of type
$I$} of the
multiset~\mcite{St} $S:=S_{n,k}:=\{1^k,2^k,\cdots, n^k\}$ to be a partition $\{B_1,\cdots,B_t\}$ of $S$ such that
\begin{enumerate}
\item $\# B_j=i_j, 1\leq j\leq t$;
\item each $B_j$ is a subset of $[n]$. In other words, $\#(B_j\cap \{j^k\})\leq 1$;
\end{enumerate}

Let $C_{I}:=C_{n,k,I}$ denote the number of the restricted
partitions of $S$ of type $I$. We will use the convention that, for
$I=\emptyset$, take $I \in \pi(0)$, $C_{I}=1$, $\ell(I)=kn$, and
$1\otimes x^{\otimes I}=1.$

\begin{proposition}\label{C1n}For $n,k \geq 1$, we have
\begin{equation*}
C_{(\underbrace{1,1,\cdots,1}_{kn})}=\frac{(kn)!}{(k!)^n} \qquad\text{and}\qquad
C_{(\underbrace{n,n, \cdots,n}_k)} = 1.
\end{equation*}
\end{proposition}

\begin{proof}
For $I \in \pi(kn)$, if $I=(\underbrace{1,1,\cdots,1}_{kn})$, and
the multiset $S=\{1^k,2^k,\cdots, n^k\}$, Then $C_{I}$ denotes the
number of ways to put the $kn$ elements in $S$ into $kn$ different
subsets, where each subset contains $1$ element. Therefore, we have
$$C_{(\underbrace{1,1,\cdots,1}_{kn})}=\frac{(kn)!}{(k!)^n}.$$

If $I=(\underbrace{n,n, \cdots, n}_k)$ and $S=\{1^k,2^k,\cdots,
n^k\}$, then $C_{I}$ denotes the number of ways to put the $kn$
elements in $S$ into $k$ different subsets, where each subset
contains $n$ different elements. There is only one way to do this, hence the second equation.
\end{proof}

\begin{proposition}\label{Cnk}For $n,k \geq 1$, we have
\begin{equation*}
\sum_{I \in \pi(kn)}C_I=C(n,k),
\end{equation*}
where $C(n,k)$ denotes the number
of ways to partition $kn$ elements in a multiset
$S=\{1^k,2^k,\cdots, n^k\}$ into several nonempty subsets of $[n]$.
\end{proposition}
\begin{proof}
For all the compositions in $\pi(kn)$, according to the definition of $C_I$, we don't need to limit the number of elements in each
subset of $[n]$.
\end{proof}

Tracing back to Cartier's work~\mcite{Ca2} forty years ago, we also recall the
description of the mixable shuffle product in terms of order
preserving maps~\cite[\S~3.1.4]{Gub} that are called {\bf stuffles} in the
study of multiple zeta values~\mcite{3BL,Br,Ho1}. We will use this interpretation in the
proof of Theorem~\mref{nk}.

For $0\leq r \leq \min(m,n)$, define
\begin{equation*}
J_{m,n,r}=\left\{(\varphi,\psi)\left|\begin{array}{l}
\varphi:[m]\rightarrow [m+n-r], \psi:[n]\rightarrow [m+n-r] \text{
are order preserving}\\
\text{injective maps and } \im(\varphi)\cup \im(\psi)=[m+n-r]
\end{array}\right.\right\}.
\end{equation*}
Let $a\in A^{\otimes m}$, $b\in A^{\otimes n}$ and $(\varphi,
\psi)\in J_{m,n,r}$. Define $a\ssha_{(\varphi, \psi)}b$ to be the
pure tensor of length $m+n-r$ whose $i$-th factor is
\begin{equation*}
(a\ssha_{(\varphi, \psi)}b)_i=\left\{
\begin{array}{ll}
a_j,& \text{if } i=\varphi(j),\ i\not\in \im\psi,\\
b_j,& \text{if } i=\psi(j),\ i\not\in \im\varphi,\\
a_jb_{j'},&\text{if } i=\varphi(j),\ i=\psi(j').
\end{array}\right.
\end{equation*}
In short,
$$a\ssha_{(\varphi,\psi)}b=a_{\varphi^{-1}(1)}b_{\psi^{-1}(1)}\otimes
\cdots \otimes a_{\varphi^{-1}(m+n-r)}b_{\psi^{-1}(m+n-r)},$$ with
the convention that $a_{\emptyset}=b_{\emptyset}=1$. Then
$a\ssha_{(\varphi, \psi)}b$ is called a \textbf{stuffle of $a$ and
$b$}. The stuffle product of $a$ and $b$ is defined by
$$a\ssha_{st,\lambda}b=\sum_{r=0}^{\min(m,n)}\lambda^r\left(\sum_{(\varphi, \psi)\in J_{m,n,r}}a\ssha_{(\varphi,\psi)}b\right)$$
which coincides with the mixable shuffle product.

\begin{theorem}\label{nk} We have the generating function
\begin{equation*}
\frac{1}{1-1\ot x^{\ot k}} =\sum_{n=0}^{\infty}(1\ot x^{\ot
k})^n=\sum_{n=0}^{\infty}\sum_{I \in \pi(kn)}C_I
\lambda^{kn-\ell(I)}(1\ot x^{\ot I}) \in \widehat{\sha}_{\lambda}(x).
\end{equation*}
\end{theorem}

\begin{proof}
According to the description of the mixable shuffle product
in terms of stuffles stated at the beginning of this
section, for $(1\ot x^{\ot k})^n$, we define an {\bf $n$-fold stuffle} to be
\begin{equation*}
\mathcal{L}_{k,n,t}=\left\{\{f_1,f_2,\cdots,f_n\}\left|\begin{array}{l}
f_i:[k]\to [t] \text{ for } 1 \leq i \leq n \text{
are order preserving}\\
\text{injective maps and }\cup_{i=1}^n \im\, f_i = [t]
\end{array}\right.\right\}.
\end{equation*}
Set
$$\overline{\mathcal{L}}_{k,n,t}=\bigcup_{t=k}^{kn}\mathcal{L}_{k,n,t}.$$
Then for each $\{f_1,f_2,\cdots,f_n\}\in
\overline{\mathcal{L}}_{k,n,t}$, we get a term in the expansion of
$(1\ot x^{\ot k})^n$ as
$$1\otimes x^{i_1} \otimes x^{i_2}\otimes \cdots \otimes x^{i_t},$$
where $$i_j=\#(\cup_{i=1}^n f^{-1}(j))\text{ for } 1\leq j \leq t.$$
Further the following relations hold.
$$1\leq i_j \leq n \text{ for } 1\leq j \leq t \text{ and }
\sum_{j=1}^ti_j=kn.$$ Therefore, each term of $(1\ot x^{\ot k})^n$
has the form $1\otimes x^{\otimes I}$, where $I$ is a composition of
$kn$ with each part less or equal to $n$.

Given $I=(i_1,\cdots,i_t)\in \pi(kn)$, we also define an {\bf  $n$-fold stuffle of type $I$} to be
\begin{equation*}
\left\{\{f_1,f_2,\cdots,f_n\}\left|\begin{array}{l} f_i:[k]\to [t]
\text{ for } 1 \leq i \leq n \text{
are order preserving injective maps,}\\
\text{such that }\cup_{i=1}^n \im f_i = [t] \text{ and }
\#(\cup_{i=1}^n f^{-1}(j)) =i_j, 1\leq j\leq t.
\end{array}\right.\right\}.
\end{equation*}
Let $S_I$ denote the number of $n$-fold stuffles of type $I$. Then
we have
$$\sum_{n=0}^\infty (1\ot x^{\ot k})^n =
\sum_{n=0}^\infty \sum_{I\in \pi(kn)} S_I \lambda^{kn-\ell(I)} (1\ot
x^{\ot I}).$$
Thus we just need to prove $C_I=S_I$.

Let $\{f_i:[k]\to [t]\}_{1\leq i\leq n}$ be a stuffle of type $I$. We can derive from $f_i$ a
$$ g: [n]\times [k]\to [t], \quad g(i,j)=g_i(j), (i,j)\in [n]\times [k].$$
Identifying $[n]\times [k]$ with the multiset $S=\{1^k,2^k,\cdots,
n^k\}$, we obtain a restricted partition of $S$ of type $I$ by
defining
$$ B_j:= g^{-1}(j), 1\leq j\leq t.$$

Conversely, given a restricted partition $\{B_j\}_{1\leq j\leq t}$
of $S$ of type $I$ and identifying it with an ordered partition of
$[n]\times [k]$. Then define
$$g:=g_B: [n]\times [k] \to [t], \quad g(i,j)=\ell \text{ if } (i,j)\in B_\ell.$$
Then it is easy to check that the maps
$$ f_i: [k]\to [t], f_i(j):=g(i,j), 1\leq i\leq n,$$
is a stuffle of type $I$. This gives the desired bijection that gives $C_I=S_I$.
\end{proof}

Based on the definition of $C_I$ and the proof of Theorem
\ref{nk}, we have the following result.
\begin{corollary}For $n,k\geq 1$, we have
\begin{equation*}
\sum_{t=k}^{kn}B(t,n,k)=C(n,k).
\end{equation*}
\end{corollary}
\begin{proof}
Given a $I\in \pi(kn)$, in the proof of Theorem \ref{nk}, we know
that $C_I=S_I$, where $S_I$ denotes the number of $n$-fold stuffles
of type $I$. Therefore, we have a equivalent illustration for $C_I$.
For any composition $I=(i_1, i_2,\cdots,i_t) \in \pi(kn)$, $C_I$
denotes the number of $n$-tuples $(T_1,T_2,\cdots,T_n)$ of subsets
$T_1, T_2, \cdots, T_n$ of $[t]$ such that $\#T_i=k$ and all the
elements in $T_1, T_2, \cdots, T_n$ are composed by $i_j$ $j's$, for
$1 \leq j \leq t$. If we take all the $t$-compositions in $\pi(kn)$,
then we don't need to consider the number of appearances of $j$ for
$1 \leq j \leq t$. Therefore, we have
\begin{equation*}
\sum_{\tiny \begin{array}{cc}I \in \pi(kn)\\
\ell(I)=t\end{array}}C_I=B(t,n,k),
\end{equation*}
where $B(t,n,k)$ is explained in
Definition \ref{B-cover}. According to Proposition \ref{Cnk}, we complete the proof.
\end{proof}

Now we consider some special cases of Theorem \ref{nk} for $\lambda=0,1$.
For $\lambda=0$, we have the following generating function.

\begin{corollary}\label{lambda0} We have
\begin{equation*}
\frac{1}{1-1\ot x^{\ot k}}=\sum_{n=0}^{\infty}(1\ot x^{\ot k})^n=\sum_{n=0}^{\infty}\frac{(kn)!}{(k!)^n}(1\ot x^{\ot (kn)})\in \widehat{\sha}_0(x).
\end{equation*}
\end{corollary}

\begin{proof}
When $\lambda=0$, Theorem \ref{nk} reduces to
\begin{equation*}
\frac{1}{1-1\ot x^{\ot k}}=\sum_{n=0}^{\infty}\sum_{\tiny \begin{array}{c} I \in \pi(kn)\\ \ell(I)=kn \end{array}}
C_I(1\ot x^{\ot I})=\sum_{n=0}^{\infty}C_{(\underbrace{1,1,\cdots,1}_{kn})}(1\ot x^{\ot (kn)})
\end{equation*}
which gives what we need by Proposition \ref{C1n}.
\end{proof}

Setting $x=1$ in Corollary \ref{lambda0}, we obtain Eq.~\eqref{bfone_kl}.

By taking the special case when $\lambda=1$ and $k=1$ in Theorem~\mref{nk}, we obtain the following result in~\mcite{Gu3}.

\begin{corollary}
We have the generating function
$${ \frac{1}{1-(1\otimes x)}=\sum_{n=0}^\infty (1\otimes x)^n=\sum_{n=0}^\infty \sum_{|I|=n} \binc{n}{I} (1\otimes x^{\otimes I})
    \in \widehat{\sha}_1(x),}$$
where for $I=\emptyset$, take $|I|=0, \binc{|I|}{I}=1, (1\otimes x)^I = 1\otimes x^{\otimes I} =
1.$
\end{corollary}

\begin{proof}
Setting $\lambda=1$ and $k=1$ in Theorem \ref{nk}, we have
\begin{equation*}
\frac{1}{1-1\ot x}
=\sum_{n=0}^{\infty}(1\ot x)^n=\sum_{n=0}^{\infty}\sum_{I \in
\pi(n)}C_I(1\ot x^{\ot I}),
\end{equation*}
where $\pi(n)$ is the set of all the compositions of $n$. For a composition $I=(i_1,\cdots,i_t) \in \pi(n)$, the
coefficient $C_{I}$ is the number of ways to divide the $n$
elements in $[n]$ into $k$ disjoint nonempty subsets, where the $j$-th subset
contains $i_j$ elements, $1\leq j \leq k$. So we get
$C_I=\binc{n}{I},$ as needed.
\end{proof}

\smallskip

\noindent {\bf Acknowledgements: } Nancy Shanshan Gu was supported by the National
Natural Science Foundation of China and the PCSIRT
Project of the Ministry of Education. Li Guo acknowledges support
from NSF grant DMS 1001855.

%========================================================================================%========================================================================================%========================================================================================%\addcontentsline{toc}{section}{\numberline {}References}

\end{document}